\documentclass[pdflatex,basic,Numbered]{jnl}% Basic Springer Nature Reference Style/Chemistry Reference Style

\usepackage{graphicx}%
\usepackage{multirow}%
\usepackage{amsmath,amssymb,amsfonts}%
\usepackage{amsthm}%
\usepackage{mathrsfs}%
\usepackage[title]{appendix}%
\usepackage{xcolor}%
\usepackage{textcomp}%
\usepackage{manyfoot}%
\usepackage{booktabs}%
\usepackage{algorithm}%
\usepackage{algorithmicx}%
\usepackage{algpseudocode}%
\usepackage{listings}%
\usepackage{bm}
\usepackage{longtable}
\usepackage{pdflscape}
\usepackage{comment}
\usepackage{booktabs}
\usepackage{multirow}
\usepackage{adjustbox}
\usepackage{booktabs}
\usepackage{caption}
\usepackage[dvipsnames]{xcolor}

\usepackage{etoolbox}

\usepackage{bm}
\usepackage{tikz}
\usepackage{tablefootnote}

\usepackage{subcaption}
\usepackage{indentfirst}
\usepackage[dvipsnames]{xcolor}
\usepackage{verbatim}
\usepackage{pdflscape}   % mejor que lscape: rota la página en el visor PDF
\usepackage{adjustbox}   % para escalar y centrar
\usepackage{float}       % para usar [H]
\usepackage{graphicx}    % necesario para \rotatebox

\usepackage[pagewise]{lineno}
\usepackage{caption} 

\usepackage{pgfplots}
\usepackage{pgfplotstable}
\usepgfplotslibrary{groupplots}
\pgfplotsset{compat=1.18}

\AtBeginEnvironment{thebibliography}{%
  \interlinepenalty=10000
  \clubpenalty=10000
  \widowpenalty=10000
}

\definecolor{boxblue}{HTML}{0072B2}
\definecolor{boxorange}{HTML}{E69F00}
\definecolor{boxgreen}{HTML}{009E73}
\definecolor{boxred}{HTML}{D55E00}
\definecolor{boxpurple}{HTML}{CC79A7}

\DeclareCaptionFont{casstyle}{\sffamily\small}
\theoremstyle{plain}
\newtheorem{lemma}{Lemma}
\newtheorem{prop}{Proposition}

\newtheorem{example}{Example}
\newtheorem*{example*}{Example}

\theoremstyle{remark}
\newtheorem{remark}{Remark}

\allowdisplaybreaks

\usepackage{enumitem}
\setlist[enumerate,1]{label=\arabic*.}
\setlist[enumerate,2]{label*=\arabic*.}
\setlist[enumerate,3]{label*=\arabic*.}

\date{\today}

\def\tsc#1{\csdef{#1}{\textsc{\lowercase{#1}}\xspace}}
\tsc{WGM}
\tsc{QE}
\tsc{EP}
\tsc{PMS}
\tsc{BEC}
\tsc{DE}
\usepackage{tikz}
\usepackage{pgfplots}
\pgfplotsset{compat=1.18}

\usepgfplotslibrary{groupplots}

\definecolor{nofair}{HTML}{D55E00} % cálido
\definecolor{fair}{HTML}{0072B2}   % frío
\definecolor{targetgray}{HTML}{777777}
\theoremstyle{thmstyleone}%
\theoremstyle{thmstylethree}%

\begin{document}
\newgeometry{
  left=3cm,
  right=3cm,
  top=4cm,
  bottom=4cm
}

\title[mode=title]{An optimization-based approach to ranking aggregation with weak order outputs}


%\author[uclm]{Juan A. Aledo}[orcid=0000-0003-1786-8087]
%\ead{JuanAngel.Aledo@uclm.es}

%\author[um]{Concepci\'on Dom\'inguez}[orcid=0000-0002-9046-4997]
%\ead{concepcion.dominguez@um.es}

%\author[umh]{Juan~de Dios Jaime-Alc\'antara}[orcid=0009-0004-4841-1295]
%\ead{jjaime@umh.es}
%\cormark[1]

%\author[umh]{Mercedes Landete}[orcid=0000-0002-5201-0476]
%\ead{landete@umh.es}

\author[1]{\fnm{Juan A.} \sur{Aledo}}\email{JuanAngel.Aledo@uclm.es}

\author[2]{\fnm{Concepci\'on} \sur{Dom\'inguez}}\email{concepcion.dominguez@um.es}

\author*[3]{\fnm{Juan~de Dios} \sur{Jaime-Alc\'antara}}\email{jjaime@umh.es}

\author[3]{\fnm{Mercedes} \sur{Landete}}\email{landete@umh.es}

\affil[1]{\orgdiv{Departamento de Matemáticas}, \orgname{Universidad de Castilla-La Mancha}, \city{Albacete}, \postcode{02071}, \country{Spain}}

\affil[2]{\orgdiv{Departamento de Estadística e Investigación Operativa}, \orgname{Universidad de Murcia}, \city{Murcia}, \postcode{30100}, \country{Spain}}

\affil[3]{\orgdiv{Departamento de Estadística, Matemáticas e Informática, Instituto Centro de Investigación Operativa}, \orgname{Universidad Miguel Hernández de Elche}, \city{Alicante}, \postcode{03202}, \country{Spain}}


\abstract{
Rank aggregation problems combine conflicting rankings of items into a single consensus ranking. In many applications, forcing all items into a strict order is too restrictive, since some items may be tied and placed in the same ordered group. This paper presents an optimization framework for rank aggregation problems in which the final ranking is a weak order, or bucket order. The framework uses binary variables to indicate whether one item is ranked before another or tied with it, and allows additional requirements to be added through linear constraints. We consider settings with an exact number of buckets, given bucket sizes, a ranking of the top items with the remaining items grouped in a final bucket, and fairness requirements for predefined groups in the upper part of the ranking. As a case study, we apply the framework to the Optimal Bucket Order Problem (OBOP), which we formulate for the first time as a mixed-integer linear programming problem. Experiments on benchmark instances derived from PrefLib and MovieLens evaluate the proposed formulation and its constrained versions. They also show that the new OBOP formulation allows us to confirm the optimality of most best-known heuristic solutions and improves some of them.
}

\keywords{Combinatorial optimization,
Rank aggregation,
Weak orders,
Optimal Bucket Order Problem, Fair ranking}



\maketitle

\section{Introduction}

Rank aggregation problems consist of combining several, possibly conflicting, rankings of the same set of items into a single consensus ordering that represents collective preferences. They arise in a wide range of applications, including social choice and information retrieval \cite{dwork2001rank,farah2007outranking}, as well as bioinformatics and recommendation systems \cite{kolde2012robust,balchanowski2023comparative}. In these contexts, heterogeneous sources provide ordinal information that must be transformed into a reliable and interpretable collective output. Designing such an output is far from trivial, since a consensus ranking should not only aggregate preferences accurately, but may also need to satisfy robustness and computational requirements \cite{dong2019alleviating,kuhlman2020rank}. Moreover, fairness considerations have recently become relevant when the resulting ranking affects the representation or visibility of different groups \cite{balestra2024fairmc}.

A broad methodological literature has been developed around rank aggregation. Different problems arise depending on the dissimilarity measure used to compare rankings and on the type of rankings allowed as inputs and outputs. Classical distances include Kendall tau, Spearman footrule, and Kemeny--Snell distances \cite{critchlow1985metric,kumar2010generalized}. Likewise, rank aggregation problems may deal with incomplete or weighted input rankings and may impose additional requirements on the consensus output, such as selecting only the most relevant items \cite{lin2010rank,chen2015spectral}. Fairness requirements have also recently been considered in order to regulate the representation of predefined groups in the resulting ranking \cite{balestra2024fairmc}. These problem choices are especially relevant when the resulting ranking is used to support decision-making processes and must therefore be both meaningful and explainable.

One of the most prominent problems in this area is the Kemeny Rank Aggregation Problem, which seeks a ranking minimizing the sum of Kendall distances to the input rankings \citep{kemeny1962mathematical}. Although this problem is NP-hard, several  integer linear programming formulations have been proposed \citep{ali2012experiments}. As well-known, integer linear programming is an optimization approach in which the decision variables are required to take integer values and both the objective function and the constraints are linear. This modelling paradigm has also been used in variants that allow ties in the consensus ranking \citep{brancotte2015rank, yoo2021new}. Exact optimization approaches are valuable not only because they provide certified optimal solutions, but also because they offer rigorous benchmarks for assessing the quality of heuristic and metaheuristic methods.

In this paper, we focus on rank aggregation problems whose output is a \emph{weak order}, also known as a \emph{bucket order}. A weak order is a complete ranking in which ties are allowed. Equivalently, the items are partitioned into ordered groups, or buckets, so that items in the same bucket are considered tied, while all items in an earlier bucket are preferred to all items in a later bucket. This type of output is natural in many decision-making contexts: items may be indistinguishable from the point of view of the available information, experts may only be able to classify alternatives into ordered categories, or the decision maker may prefer a grouped ranking rather than an artificial strict ordering.

While optimization models for linear-order aggregation are well established (see \cite{alcaraz2022rank}), the corresponding literature on weak-order aggregation remains comparatively limited. Existing studies on linear ordering have explored alternative objective functions to reduce the multiplicity of optimal solutions \citep{benito2025multiple}, bilevel formulations to obtain consistent orderings across sets of elements \citep{labbe2023bilevel}, and structural properties of ranking models \citep{ceberio2015linear, martin2025improving}. However, these advances have not yet been extended into a unified optimization framework for weak-order aggregation. To the best of our knowledge, no existing approach simultaneously accommodates exact optimization, structural restrictions on the bucket structure, upper-part ranking requirements, and group-based fairness constraints. This paper addresses this gap by developing a mixed-integer linear programming framework for rank aggregation with weak-order outputs. 

% Mixed-integer linear programming extends integer linear programming by allowing some variables to be integer or binary and others to be continuous. The proposed framework is based on pairwise binary variables that indicate, for every pair of items, whether one item is ranked before the other or whether the two items are tied. Once this weak-order structure is encoded, additional modelling requirements can be incorporated by adding linear constraints, without changing the general architecture of the formulation.

We study settings in which the structure of the consensus bucket order is partially prescribed. Specifically, we address cases where the ranking must contain a fixed number of nonempty buckets, as well as variants in which bucket sizes are predetermined, either uniformly or according to a given sequence. We also consider situations where the decision maker is primarily interested in the upper part of the ranking. In such cases, the top $k$ items are selected and ordered as a weak order, whereas the remaining $n-k$ items are grouped into a single final bucket. This approach preserves meaningful distinctions among the highest-ranked items while treating lower-ranked alternatives as interchangeable, thereby avoiding the ambiguity associated with first constructing a complete bucket order and subsequently truncating it.

In addition, we incorporate group-based fairness requirements. In many applications, items belong to predefined groups, and the consensus ranking should regulate the representation of these groups among the most visible positions. We model this by imposing lower and upper proportional bounds over cumulative prefixes of buckets. In this way, fairness is enforced over the ordered structure of the ranking, rather than independently within each bucket, while remaining compatible with the presence of ties.

A central application of the framework is the Optimal Bucket Order Problem (OBOP). In this problem, the input is a pairwise order matrix that describes the degree to which each item is preferred to each other item, and the aim is to find a bucket order whose associated bucket matrix is as close as possible to that input matrix. The OBOP is NP-hard, as established in \cite{gionis2006algorithms}. 
Early algorithmic approaches were based on randomized and greedy procedures 
\citep{ukkonen2009randomized,kenkre2011discovering}. Subsequent works proposed 
improved greedy strategies \citep{aledo2017utopia}, evolution-based methods 
\citep{aledo2018approaching}, and scalable algorithms for large instances 
\citep{aledo2021highly}. Other metaheuristic approaches have also been explored 
in \cite{lorena2021biased}. We provide, to the best of our knowledge, the first 
 linear programming formulation for this problem. This makes 
it possible to find optimal solutions and, as an additional result, to assess 
whether previously reported heuristic solutions are optimal or can be improved.

The proposed formulations are evaluated through computational experiments on benchmark instances derived from PrefLib and MovieLens. The experiments analyze the behaviour of the constrained variants, showing how requirements on the number or size of buckets, upper-ranking considerations, and fairness constraints affect both the optimal bucket order and the computational performance of the models. In addition, the computational results certify the optimality of most of the best-known heuristic solutions for the OBOP, improve the current best-known solutions for several instances, and provide exact benchmarks for future algorithmic developments.

Overall, the contribution of this paper is twofold. From a modelling perspective, it provides a unified exact optimization framework for rank aggregation with weak-order outputs and shows how several practically relevant requirements can be represented through linear constraints. From a computational perspective, it introduces the first exact approach for the OBOP and uses it to certify optimality, improve previous heuristic results, and study the effect of structural and fairness constraints on benchmark instances.

The remainder of the paper is organized as follows. Section~\ref{sec:IPF} introduces the integer programming framework for weak-order aggregation. Section~\ref{Sec:OBOP} defines the Optimal Bucket Order Problem and presents its exact formulation. Sections~\ref{sec:numero_fijo}, \ref{Sec:TCU}, and \ref{sec:fairness} develop, respectively, the models with a fixed number of buckets, the formulation for selecting an upper part of the ranking while collapsing the remaining items into a final bucket, and the fairness-aware weak-order aggregation model. Section~\ref{sec:CE} reports the computational experiments. Section~\ref{sec:CFR} summarizes the main conclusions and outlines future research directions. Additional technical details are provided in the appendix.

\section{Integer programming framework} \label{sec:IPF}

We introduce the integer programming framework used throughout the paper to model aggregation problems whose feasible solutions are weak orders, following \cite{fiorini2004weak}. A \emph{weak order} (or \emph{bucket order}) is a complete preorder: every pair of items is comparable, ties are allowed, and the relation is transitive. Hence, items may be grouped into ordered equivalence classes, also called buckets.

Let \( [n]=\{1,2,\dots,n\} \) be the set of items. We encode a weak order through the binary variables
\[
x_{rs}=
\begin{cases}
1 & \text{if } r \text{ comes before or is tied with } s,\\
0 & \text{otherwise},
\end{cases}
\qquad \forall r,s\in [n]:\ r\neq s.
\]
Thus, \(x_{rs}=x_{sr}=1\) represents a tie between \(r\) and \(s\), whereas \(x_{rs}=1\) and \(x_{sr}=0\) means that item \(r\) is ranked ahead of item \(s\).

The weak-order structure is enforced by the following constraints \citep{brancotte2015rank,yoo2021new}:
\begin{subequations} \label{mod:empates}
    \begin{align} 
    \min\quad & f(\Pi,\bm{x}) \\
        \text{s.t.}\quad 
        & x_{rs}+x_{sr}\ge 1 
        && \forall r,s\in [n]:\ r<s, \label{eq:antes_o_despues_1} \\
        & x_{rs}+x_{st} \le 1+x_{rt}
        && \forall r,s,t\in [n]:\ r \neq s \neq t \neq r, \label{eq:transitividad_1} \\
        & x_{rs}\in\{0,1\}
        && \forall r,s\in [n]:\ r\neq s. \label{eq:naturaleza_x_1}
    \end{align}
\end{subequations}

The objective function \(f(\Pi,\bm{x})\) depends on the particular aggregation problem \(\Pi\). Constraint~\eqref{eq:antes_o_despues_1} ensures comparability, while constraint~\eqref{eq:transitividad_1} enforces transitivity. Therefore, every feasible solution of model~\eqref{mod:empates} represents a weak order.

Although this paper focuses on weak orders, linear orders can be obtained by replacing the inequality in constraint~\eqref{eq:antes_o_despues_1} with an equality.

Finally, the framework also allows us to detect alternative optimal solutions. Given a previously found solution \(\hat{\bm{x}}\), it can be excluded from the feasible region by adding
\[
1 \le 
\sum_{r=1}^{n}
\sum_{\substack{s=1\\s\neq r}}^{n}
\left((1-\hat{x}_{rs})x_{rs}
+\hat{x}_{rs}(1-x_{rs})\right).
\]
This constraint forces at least one binary variable to differ from \(\hat{\bm{x}}\). Thus, by iteratively resolving the model and excluding each newly identified optimal solution, we can determine whether alternative optimal solutions exist and enumerate them whenever they do. Throughout the paper, whenever we report that an optimum is unique or that all alternative optimal solutions have been identified, this conclusion is obtained using this iterative exclusion procedure. To avoid repetition, we do not restate this verification in each example.

\section{The Optimal Bucket Order Problem}\label{Sec:OBOP}

In this section, we show how the general framework introduced in Section~\ref{sec:IPF} can be particularized to the \emph{Optimal Bucket Order Problem} (OBOP). This leads, to the best of our knowledge, to the first exact mixed-integer linear programming formulation proposed for the OBOP.

A bucket order \(\mathcal{B}\) on \(n\) items is an ordered partition \(B_1,B_2,\dots,B_p\) of \( [n]=\{1,2,\dots,n\} \), with \(1\le p\le n\). We represent \(\mathcal{B}\) by listing the elements with bars between consecutive buckets. For instance,
\[
4 \mid 1\ 3 \mid 2\ 5
\]
denotes a bucket order on \( [5] \) with three buckets: item \(4\) is ranked ahead of all the others, items \(1\) and \(3\) are tied, and items \(2\) and \(5\) are tied in the last bucket.

Each bucket order \(\mathcal{B}\) can be represented by a \emph{bucket matrix} \(B=(b_{rs})_{n\times n}\), where
\[
b_{rs}=
\begin{cases}
1 & \text{if } r \text{ is preferred to } s,\\
0.5 & \text{if } r \text{ and } s \text{ are tied},\\
0 & \text{if } s \text{ is preferred to } r.
\end{cases}
\]
Thus, \(b_{rr}=0.5\) for all \(r\in[n]\), and \(b_{rs}+b_{sr}=1\) for all \(r\neq s\).

The input of the OBOP is a \emph{pair order matrix} \(C=(c_{rs})_{n\times n}\), satisfying \(c_{rr}=0.5\) for all \(r\in[n]\),  \(c_{rs}\in[0,1]\), and \(c_{rs}+c_{sr}=1\) for all \(r,s\in[n]\) with \(r\neq s\). Every bucket matrix is a pair order matrix, but not every pair order matrix corresponds to a bucket order, since transitivity may fail and its entries are not restricted to the values \(0\), \(0.5\), and \(1\).

The \emph{Optimal Bucket Order Problem} \citep{gionis2006algorithms,ukkonen2009randomized} consists of finding a bucket order \(\mathcal{B}\), or equivalently its bucket matrix \(B\), minimizing
\[
D(B,C)=
\sum_{r=1}^{n}
\sum_{s=1}^{n}
|b_{rs}-c_{rs}|.
\]

Since the OBOP is NP-hard \citep{gionis2006algorithms}, several greedy and metaheuristic algorithms have been proposed \citep{ukkonen2009randomized,kenkre2011discovering,aledo2017utopia,aledo2018approaching,d2019median,aledo2021highly}. However, no exact optimization formulation has been proposed so far.

\subsection{An exact mixed-integer linear programming formulation for the {\rm OBOP}} \label{sec:ILP_OBOP}

As established in Section~\ref{sec:IPF}, the constraints of model~\eqref{mod:empates} guarantee that every feasible solution represents a weak order. Therefore, to formulate the OBOP within this framework, it remains to express the distance \(D(B,C)\) in terms of the variables \(x_{rs}\). The following proposition provides this link.

\begin{prop}
Let \(\bm{x}\) be any feasible solution of model~\eqref{mod:empates}, and let \(B\) be the bucket matrix induced by the corresponding weak order. Then, for every \(r,s\in[n]\) with \(r\neq s\),

\[
b_{rs}=
\frac{x_{rs}-x_{sr}+1}{2}
\qquad
\forall r,s\in[n]:\ r\neq s.
\]
Consequently, the {\rm OBOP} objective can be written as
\begin{equation}\label{eq:objetivo_valor_absoluto_red}
f({\rm OBOP},\bm{x}) =
2\sum_{r=1}^{n-1}\sum_{s=r+1}^{n}
\left|
\frac{x_{rs}-x_{sr}+1}{2}
-c_{rs}
\right|.
\end{equation}
\end{prop}

\begin{proof}
For any pair of distinct items \(r\) and \(s\), there are only three possible relations in a weak order: \(r\) precedes \(s\), \(r\) and \(s\) are tied, or \(s\) precedes \(r\). According to the definition of the variables \(\bm{x}\), the expression \((x_{rs}-x_{sr}+1)/2\) takes the values \(1\), \(0.5\), and \(0\), respectively. These are precisely the corresponding values of \(b_{rs}\), and hence \(b_{rs}=(x_{rs}-x_{sr}+1)/2\).

Since \(b_{rr}=c_{rr}=0.5\), the diagonal terms in \(D(B,C)\) vanish. Moreover, using \(b_{sr}=1-b_{rs}\) and \(c_{sr}=1-c_{rs}\), the contributions of \((r,s)\) and \((s,r)\) coincide. Therefore, the {\rm OBOP} objective can be written as twice the sum over pairs with \(r<s\), yielding~\eqref{eq:objetivo_valor_absoluto_red}.
\end{proof}

To linearize~\eqref{eq:objetivo_valor_absoluto_red}, we introduce continuous auxiliary variables \(d_{rs}\), for all \(r,s\in[n]\) with \(r<s\), representing the absolute deviation between \(b_{rs}\) and \(c_{rs}\). The objective function becomes
\begin{equation}\label{eq:objetivo_linealizado}
f({\rm OBOP},\bm{x})=
2\sum_{r=1}^{n}\sum_{s=r+1}^{n} d_{rs},
\end{equation}
together with the linear constraints
\begin{subequations}\label{eq:linearizacion_OBOP}
    \begin{align}
        d_{rs} &\ge 
        \frac{x_{rs}-x_{sr}+1}{2}-c_{rs}
        && \forall r,s\in[n]:\ r<s, \label{eq:lin1_1} \\
        d_{rs} &\ge 
        c_{rs}-\frac{x_{rs}-x_{sr}+1}{2}
        && \forall r,s\in[n]:\ r<s. \label{eq:lin2_1}
    \end{align}
\end{subequations}

Combining the weak-order constraints with this linearization yields the following exact mixed-integer linear programming formulation for the OBOP:
\begin{subequations}\label{mod:OBOP}
    \begin{align} 
    \min\quad 
        & 2\sum_{r=1}^{n}\sum_{s=r+1}^{n} d_{rs} \label{eq:obj_OBOP} \\
    \text{s.t.}\quad
        & \eqref{eq:antes_o_despues_1},\ \eqref{eq:transitividad_1},\ \eqref{eq:lin1_1},\ \eqref{eq:lin2_1}, \\
        & x_{rs}\in\{0,1\}
        && \forall r,s\in [n]:\ r\neq s. \label{eq:naturaleza_x_OBOP}
    \end{align}
\end{subequations}

Constraints~\eqref{eq:antes_o_despues_1} and~\eqref{eq:transitividad_1} enforce the weak-order structure, while constraints~\eqref{eq:lin1_1} and~\eqref{eq:lin2_1} linearize the absolute deviations in the OBOP objective. Therefore, model~\eqref{mod:OBOP} solves the OBOP exactly.

\begin{remark}
The factor \(2\) in objective function~\eqref{eq:obj_OBOP} could be omitted without changing the set of optimal solutions. We keep it in order to preserve the objective-value scale used in the OBOP literature.
\end{remark}

We conclude this section with a small instance that will be used as a running example throughout the paper to illustrate the effect of the different constraints introduced later.

\begin{example}\label{ex:ejemplo_1}
We consider the following instance of the {\rm OBOP} with \(8\) items, which will be used as a running example throughout the paper:
\[
C=\frac{1}{100}
\begin{pmatrix}
50 & 52 & 46 & 72 & 60 & 70 & 82 & 90 \\
48 & 50 & 10 & 42 & 60 & 90 & 22 & 78 \\
54 & 90 & 50 & 76 & 98 & 85 & 82 & 80 \\
28 & 58 & 24 & 50 & 80 & 76 & 65 & 80 \\
40 & 40 &  2 & 20 & 50 & 55 &  0 & 20 \\
30 & 10 & 15 & 24 & 45 & 50 & 50 &  0 \\
18 & 78 & 18 & 35 &100 & 50 & 50 & 90 \\
10 & 22 & 20 & 20 & 80 &100 & 10 & 50
\end{pmatrix}.
\]
Solving this instance with model~\eqref{mod:OBOP} yields the unique optimum
\[
{\rm 1\ 3\ \mid \ 2\ 4\ 7\ \mid \ 8\ \mid \ 5\ 6}
\]
with an objective value of \(10.78\).
\end{example}

\section{Weak-order aggregation under additional constraints}

In this section, we extend the weak-order aggregation framework by incorporating additional constraints on the structure and composition of the consensus ranking. These constraints allow the decision maker to prescribe the number or size of buckets, focus on the top-ranked items, or enforce group-based fairness requirements. Although the models are stated for a general aggregation objective \(f(\Pi,\bm{x})\), they can be directly particularized to the OBOP using the objective function and linearization introduced in Section~\ref{sec:ILP_OBOP}. 

\subsection{Models with a fixed number \texorpdfstring{$p$}{p} of buckets}
\label{sec:numero_fijo}

In this section, we extend the weak-order framework to the case where the consensus ranking is required to contain exactly \(p\) nonempty buckets, with \(1\le p\le n\). We refer to this setting as the \(p\)-bucket problem. A related heuristic median-ranking problem with a fixed number of buckets was considered in \cite{d2019median}; here, we develop exact formulations within a general optimization framework.

This setting arises naturally in classification and rating systems where items are grouped into a limited number of ordered categories, such as tiered university rankings, rating scales, or award levels. We present two alternative formulations: one based on item--bucket assignment variables and another based on representative items. The objective function is left general as \(f(\Pi,\bm{x})\), so both formulations can be particularized to different aggregation problems.

For the first model, we define
\[
y_{ru}=
\begin{cases}
1 & \text{if item } r \text{ is assigned to bucket } u,\\
0 & \text{otherwise},
\end{cases}
\qquad
\forall r\in[n],\ \forall u\in[p].
\]
Here, bucket \(u\) represents the \(u\)-th position in the weak order. The assignment-based formulation is given by
\begin{subequations} \label{mod:p-bucket_1}
    \begin{align} 
    \min\quad & f(\Pi,\bm{x}) \\
    \text{s.t.}\quad 
        & \sum_{u=1}^p y_{ru}=1
        && \forall r\in[n], \label{eq:elemento_en_bucket_3} \\
        & \sum_{r=1}^n y_{ru}\ge 1
        && \forall u\in[p], \label{eq:buckets_no_vacios_3} \\
        & y_{ru}+y_{su}\le x_{rs}+x_{sr}
        && \forall r,s\in[n],\ \forall u\in[p]:\ r<s, \label{eq:empates_bucket_3} \\
        & x_{sr}+\sum_{v=1}^{u}y_{rv}
        +\sum_{v=u+1}^{p}y_{sv}\le 2
        && \forall r,s\in[n],\ \forall u\in[p-1]:\ r\neq s, \label{eq:buckets_ordenados_3} \\
        & x_{rs}\in\{0,1\}
        && \forall r,s\in[n]:\ r\neq s, \label{eq:naturaleza_x_3} \\
        & y_{ru}\in\{0,1\}
        && \forall r\in[n],\ \forall u\in[p]. \label{eq:naturaleza_y_3}
    \end{align}
\end{subequations}

Constraints~\eqref{eq:elemento_en_bucket_3} assign each item to exactly one bucket, while constraints~\eqref{eq:buckets_no_vacios_3} ensure that all \(p\) buckets are nonempty. Constraints~\eqref{eq:empates_bucket_3} force items assigned to the same bucket to be tied, and constraints~\eqref{eq:buckets_ordenados_3} ensure consistency between the order of the buckets and the pairwise variables \(x_{rs}\). In particular, if item \(r\) is assigned to bucket \(u\) or to an earlier bucket, and item \(s\) is assigned to a later bucket, then \(s\) cannot be ranked ahead of \(r\).

The assignment-based formulation can be easily adapted when the size of each bucket is known in advance. Since bucket \(u\) explicitly represents the \(u\)-th position in the weak order, prescribed bucket sizes can be imposed by replacing constraints~\eqref{eq:buckets_no_vacios_3} with
\begin{equation}\label{eq:prescribed_bucket_sizes_assignment}
    \sum_{r=1}^n y_{ru}=q_u
    \qquad \forall u\in[p],
\end{equation}
where \(q_u\) denotes the prescribed size of bucket \(u\), with \(q_u\ge 1\) and \(\sum_{u=1}^p q_u=n\). The equal-size case is obtained as the particular case \(q_u=q\) for all \(u\in[p]\). The theoretical properties stated below remain valid under this replacement.

For completeness, we note that the constraints from the base weak-order model~\eqref{mod:empates} are redundant in this formulation; they remain valid inequalities for model~\eqref{mod:p-bucket_1}. We formalize this result below.

\begin{prop}\label{prop:des_valida_1}
The base inequalities~\eqref{eq:antes_o_despues_1} and~\eqref{eq:transitividad_1}, defining the weak-order structure, are valid inequalities for model~\eqref{mod:p-bucket_1}.
\end{prop}

\begin{proof}
Let \(r,s\in[n]\) with \(r<s\). By~\eqref{eq:elemento_en_bucket_3}, there exists some \(u\in[p]\) such that \(y_{ru}=1\). Applying~\eqref{eq:empates_bucket_3} to \((r,s,u)\) yields \(1+y_{su}\le x_{rs}+x_{sr}\), and therefore \(x_{rs}+x_{sr}\ge 1\). Hence~\eqref{eq:antes_o_despues_1} holds.

For transitivity, let \(r,s,t\in[n]\) be pairwise distinct and assume that \(x_{rs}=x_{st}=1\). By~\eqref{eq:buckets_ordenados_3}, item \(r\) cannot be placed in a bucket after the bucket of \(s\), and item \(s\) cannot be placed in a bucket after the bucket of \(t\). Hence, item \(r\) cannot be placed in a bucket after the bucket of \(t\). If \(r\) and \(t\) belong to the same bucket, then they are tied and \(x_{rt}=1\) by~\eqref{eq:empates_bucket_3}. If the bucket of \(r\) is strictly before the bucket of \(t\), then \(r\) is ranked ahead of \(t\), so \(x_{tr}=0\) by~\eqref{eq:buckets_ordenados_3}, and the comparability inequality already proved gives \(x_{rt}=1\). Thus~\eqref{eq:transitividad_1} is satisfied.
\end{proof}

Moreover, the integrality of variables \(\bm{x}\) can be relaxed, as formalized below.

\begin{prop}\label{prop:relax_x_assignment}
The integrality constraints~\eqref{eq:naturaleza_x_3} on variables \(\bm{x}\) can be relaxed.
\end{prop}

\begin{proof}
Let \(r,s\in[n]\), with \(r\neq s\), and let \(u,v\in[p]\) be the buckets assigned to \(r\) and \(s\), respectively. If \(u=v\), then~\eqref{eq:empates_bucket_3}, together with the bounds of the linear relaxation, forces \(x_{rs}=x_{sr}=1\). If \(u<v\), constraint~\eqref{eq:buckets_ordenados_3} implies \(x_{sr}=0\), and by the valid inequality~\eqref{eq:antes_o_despues_1} we obtain \(x_{rs}=1\). The case \(u>v\) is symmetric. Hence, once the assignment variables are fixed, all variables \(x_{rs}\) are forced to take integer values.
\end{proof}

The above formulation fully characterizes the \(p\)-bucket problem using assignment variables. 
However, an alternative model can be obtained by introducing representative-based variables, 
a modelling device that has also been used in other combinatorial optimization problems involving 
partitions into classes, such as vertex coloring \citep{campelo2008asymmetric}. 
In our setting, representatives provide a different way of encoding the buckets of a weak order.

In this second formulation, we assume that the representative of each bucket is the item with the 
highest index among its members. Accordingly, if item \(r\) belongs to a bucket whose representative 
is \(s\), then \(r < s\). Based on this convention, instead of using the assignment variables \(y_{ru}\), 
we work with the following decision variables:
\[
\alpha_r=
\begin{cases}
1 & \text{if item } r \text{ is chosen as the representative of a bucket},\\
0 & \text{otherwise},
\end{cases}
\qquad \forall r\in[n],
\]
and
\[
\beta_{rs}=
\begin{cases}
1 & \text{if item } r \text{ belongs to the bucket represented by } s,\\
0 & \text{otherwise},
\end{cases}
\qquad \forall r,s\in[n]:\ r<s.
\]

The representative-based formulation is
\begin{subequations} \label{mod:p-bucket_2}
    \begin{align} 
    \min\quad & f(\Pi,\bm{x}) \\
    \text{s.t.}\quad 
        & \eqref{eq:antes_o_despues_1},\ \eqref{eq:transitividad_1} \\
        & \sum_{r=1}^n \alpha_r=p, \label{eq:p_buckets_4} \\
        & \beta_{rs}\le \alpha_s
        && \forall r,s\in[n]:\ r<s, \label{eq:asigniado_si_representante_4} \\
        & \sum_{s=r+1}^n \beta_{rs}+\alpha_r=1
        && \forall r\in[n], \label{eq:asignado_o_representante_4} \\
        & x_{rs}\ge \beta_{rs}
        && \forall r,s\in[n]:\ r<s, \label{eq:antes_si_mismo_bucket_4} \\
        & x_{sr}\ge \beta_{rs}
        && \forall r,s\in[n]:\ r<s, \label{eq:despues_si_mismo_bucket_4} \\
        & x_{rs}+x_{sr}+\alpha_r\le 2
        && \forall r,s\in[n]:\ r<s, \label{eq:un_representante_por_bucket_4} \\
        & x_{rs}\in\{0,1\}
        && \forall r,s\in[n]:\ r\neq s, \label{eq:naturaleza_x_4} \\
        & \alpha_r\in\{0,1\}
        && \forall r\in[n], \label{eq:naturaleza_alpha_4} \\
        & \beta_{rs}\in\{0,1\}
        && \forall r,s\in[n]:\ r<s. \label{eq:naturaleza_beta_4}
    \end{align}
\end{subequations}

Constraint~\eqref{eq:p_buckets_4} ensures that exactly \(p\) representatives are selected and therefore \(p\) buckets are created. Constraints~\eqref{eq:asigniado_si_representante_4} enforce that an item can only be assigned to a selected representative. Constraints~\eqref{eq:asignado_o_representante_4} require that every item is either a representative itself or is assigned to the bucket of a representative with larger index. Constraints~\eqref{eq:antes_si_mismo_bucket_4} and~\eqref{eq:despues_si_mismo_bucket_4} ensure that, whenever an item belongs to a bucket represented by another one, both items are tied. By transitivity, this implies that all items within the same bucket are tied. Finally, constraints~\eqref{eq:un_representante_por_bucket_4} guarantee that each bucket has exactly one representative, namely the largest-indexed item in that bucket.

In addition, we note that the following inequality, although not explicitly included in model~\eqref{mod:p-bucket_2}, is always satisfied by any feasible solution and can therefore be regarded as a valid inequality.

\begin{prop}\label{prop:empate_mismo_bucket_4}
The following constraints are valid inequalities for model~\eqref{mod:p-bucket_2}:
\begin{equation}
\beta_{rs}+\alpha_s\le x_{rs}+x_{sr}
\qquad \forall r,s\in[n]:\ r<s.
\label{eq:empate_mismo_bucket_4}
\end{equation}
\end{prop}

\begin{proof}
If \(\beta_{rs}+\alpha_s\le 1\), the result follows from~\eqref{eq:antes_o_despues_1}. If \(\beta_{rs}+\alpha_s=2\), then \(\beta_{rs}=\alpha_s=1\), and by~\eqref{eq:antes_si_mismo_bucket_4}--\eqref{eq:despues_si_mismo_bucket_4}, we obtain \(x_{rs}=x_{sr}=1\). Hence~\eqref{eq:empate_mismo_bucket_4} holds in all cases.
\end{proof}

\begin{prop}\label{prop:alternative_representative}
By adding constraint~\eqref{eq:empate_mismo_bucket_4} to model~\eqref{mod:p-bucket_2} and removing constraints~\eqref{eq:antes_si_mismo_bucket_4} and~\eqref{eq:despues_si_mismo_bucket_4}, we obtain an alternative formulation of the problem.
\end{prop}

\begin{proof}
Fix \(r<s\). If \(\beta_{rs}=0\), then the removed constraints hold trivially. If \(\beta_{rs}=1\), constraint~\eqref{eq:asigniado_si_representante_4} implies \(\alpha_s=1\), and~\eqref{eq:empate_mismo_bucket_4} gives \(x_{rs}=x_{sr}=1\), which satisfies both~\eqref{eq:antes_si_mismo_bucket_4} and~\eqref{eq:despues_si_mismo_bucket_4}. Hence, the added constraint reproduces the effect of the removed ones.
\end{proof}

The representative-based formulation can also be adapted to the equal-size case. If all buckets are required to contain exactly \(q\) items, with \(pq=n\), it is enough to add
\begin{equation}\label{eq:equal_bucket_sizes_representative}
    \sum_{r=1}^{s-1}\beta_{rs}=(q-1)\alpha_s
    \qquad \forall s\in[n].
\end{equation}
Indeed, if item \(s\) is selected as a representative, then exactly \(q-1\) additional items must be assigned to its bucket; otherwise, no item can be assigned to \(s\).

Unlike the assignment-based formulation, the representative-based model does not index buckets by rank position. Therefore, arbitrary prescribed sizes \((q_1,\dots,q_p)\) cannot be imposed directly without introducing additional variables linking representatives to ordered bucket positions.

To illustrate the effect of prescribing the bucket structure, we revisit the running example introduced in Example~\ref{ex:ejemplo_1}.

\begin{example}\label{ex:ejemplo_2}
Continuing with Example~\ref{ex:ejemplo_1}, we consider several constrained versions of the {\rm OBOP}. First, suppose that only the number of buckets is fixed. Adapting models~\eqref{mod:p-bucket_1} and~\eqref{mod:p-bucket_2} to the {\rm OBOP}, for \(p=2\) and \(p=5\) we obtain, respectively,
\[
{\rm 1\ 2\ 3\ 4\ 7 \mid 5\ 6\ 8}
\qquad\text{and}\qquad
{\rm 1\ 3 \mid 4\ 7 \mid 2 \mid 8 \mid 5\ 6},
\]
with objective values \(12.14\) and \(11.34\). When \(p=4\), the solution coincides with the unconstrained {\rm OBOP} optimum, as expected.

If equal bucket sizes are imposed, the optima for two buckets of capacity four and four buckets of capacity two are
\[
{\rm 1\ 3\ 4\ 7 \mid 2\ 5\ 6\ 8}
\qquad\text{and}\qquad
{\rm 1\ 3 \mid 4\ 7 \mid 2\ 8 \mid 5\ 6},
\]
with objective values \(13.14\) and \(11.46\), respectively. Thus, fixing not only the number of buckets but also their sizes may change the optimal bucket order.

Finally, using the assignment-based formulation, we may prescribe different bucket sizes by rank position. For capacities \((1,3,4)\) and \((1,2,2,3)\), the optima are
\[
{\rm 3 \mid 1\ 4\ 7 \mid 2\ 5\ 6\ 8}
\qquad\text{and}\qquad
{\rm 3 \mid 1\ 4 \mid 2\ 7 \mid 5\ 6\ 8},
\]
with objective values \(13.66\) and \(13.78\). For capacities \((2,3,1,2)\), which match the bucket sizes of the unconstrained optimum, the solution coincides with the unconstrained {\rm OBOP} optimum, as expected.
\end{example}

\subsection{Tail-Collapsed Upper-\texorpdfstring{$k$}{k}}
\label{Sec:TCU}

In many applications, identifying the most relevant \(k\) items is more informative than producing a complete ranking of all \(n\) items. In a weak-order setting, however, extracting a top-\(k\) list from a complete solution may be ambiguous, since the boundary between the selected and non-selected items may cut through a bucket of tied elements. Several approaches to top-\(k\) rank aggregation have been studied in the literature; see, for instance, \cite{lin2010rank}.

We address this issue by introducing the Tail-Collapsed Upper-\(k\) problem, with \(k \in [n-1]\). The idea is to force the last \(n-k\) items to form a single final bucket, while the upper part, containing exactly \(k\) items, remains free to adopt any weak-order structure. Thus, the model does not simply truncate a complete ranking, nor does it force the top \(k\) items to be tied in a single bucket. Instead, it selects the best weak order among those in which the lower \(n-k\) items are treated as interchangeable.

This approach is therefore different from both the classical top-\(k\) problem and a two-bucket formulation with a first bucket of size \(k\). A truncation-based method first solves the full problem and then keeps the first \(k\) positions, possibly requiring an arbitrary tie-breaking rule. A two-bucket formulation, on the other hand, forces the selected \(k\) items to form a single undifferentiated group. The Tail-Collapsed Upper-\(k\) formulation preserves the internal structure of the selected items while collapsing only the tail.

To formulate the problem, we introduce the variables
\[
z_r=
\begin{cases}
1 & \text{if item } r \text{ belongs to the last bucket},\\
0 & \text{otherwise},
\end{cases}
\qquad \forall r\in[n].
\]

The formulation is given by
\begin{subequations} \label{mod:top-k}
    \begin{align} 
    \min\quad & f(\Pi,\bm{x}) \\
    \text{s.t.}\quad 
        & \eqref{eq:antes_o_despues_1},\ \eqref{eq:transitividad_1} \\
        & \sum_{r\in[n]} z_r = n-k, \label{eq:ultimo_bucket_7} \\
        & x_{rs}+z_r \le 1+z_s
        && \forall r,s\in[n]: r\neq s, \label{eq:ultimo_bucket_empate_7} \\
        & x_{rs}\ge z_s
        && \forall r,s\in[n]: r\neq s, \label{eq:ultimo_bucket_detras_7} \\
        & x_{rs}\in\{0,1\}
        && \forall r,s\in[n]: r\neq s, \label{eq:naturaleza_x_7} \\
        & z_r\in\{0,1\}
        && \forall r\in[n]. \label{eq:naturaleza_z_7}
    \end{align}
\end{subequations}

Constraint~\eqref{eq:ultimo_bucket_7} fixes the size of the last bucket to \(n-k\). Constraints~\eqref{eq:ultimo_bucket_empate_7} and~\eqref{eq:ultimo_bucket_detras_7} ensure that every item in the tail is tied with the other tail items and is ranked strictly after every item outside the tail.

We illustrate the formulation using the running example introduced in Example~\ref{ex:ejemplo_1}.

\begin{example}\label{ex:ejemplo_6}
Using the same data as in Example~\ref{ex:ejemplo_1}, we solve the {\rm TCU-\(k\) OBOP}, obtained by adapting formulation~\eqref{mod:top-k} to the {\rm OBOP}. For \(k=4\) and \(k=7\), the optima are
\[
{\rm 1\ 3 \mid 4\ 7 \mid\mid 2\ 5\ 6\ 8}
\qquad\text{and}\qquad
{\rm 1\ 3 \mid 2\ 4\ 7 \mid 8 \mid 5 \mid\mid 6},
\]
with objective values \(12.66\) and \(11.58\), respectively. Hereafter, the symbol \(\mid\mid\) denotes the separation between the upper part, containing \(k\) items, and the collapsed tail. For \(k=6\), the model reproduces the unique optimal ranking of the unconstrained {\rm OBOP}.
\end{example}

The following example shows that the Tail-Collapsed Upper-\(k\) problem is not equivalent to truncating an optimal OBOP solution, nor to imposing two fixed buckets.

\begin{example}\label{ex:contraejemplo_top-k}
Consider an instance with \(10\) items and pair order matrix
\[
C=\frac{1}{100}
\begin{pmatrix}
  50 & 82 & 12 & 77 &  5 & 91 & 24 & 80 & 15 & 76 \\
  18 & 50 & 88 & 22 & 79 & 10 &  7 & 95 & 25 & 79 \\
  88 & 12 & 50 & 81 & 23 & 90 & 75 & 14 &  5 & 83 \\
  23 & 78 & 19 & 50 & 80 &  0 & 25 & 85 & 11 & 97 \\
  95 & 21 & 77 & 20 & 50 &  8 & 89 & 76 &  2 & 75 \\
   9 & 90 & 10 &100 & 92 & 50 & 78 & 18 & 84 & 20 \\
  76 & 93 & 25 & 75 & 11 & 22 & 50 & 79 &  9 & 85 \\
  20 &  5 & 86 & 15 & 24 & 82 & 21 & 50 & 77 & 13 \\
  85 & 75 & 95 & 89 & 98 & 16 & 91 & 23 & 50 & 80 \\
  24 & 21 & 17 &  3 & 25 & 80 & 15 & 87 & 20 & 50
\end{pmatrix}.
\]

Solving the unconstrained {\rm OBOP} yields exactly two optima,
\[
{\rm 6 \mid 9 \mid 5 \mid 3\ 7 \mid 1 \mid 4 \mid 2 \mid 10 \mid 8}
\qquad\text{and}\qquad
{\rm 6 \mid 9 \mid 5 \mid 3 \mid 7 \mid 1 \mid 4 \mid 2 \mid 10 \mid 8},
\]
both with objective value \(13.14\). The only difference is whether items \(3\) and \(7\) are tied or strictly ordered.

For \(k=6\), the {\rm TCU-\(k\) OBOP} yields
\[
{\rm 9 \mid 3 \mid 1\ 6\ 7 \mid 4 \mid\mid 2\ 5\ 8\ 10}
\]
with objective value \(28.9\). This solution does not contain the same upper set as the unconstrained {\rm OBOP} optima. Indeed, the best rankings whose collapsed tail coincides with the tail obtained by truncating the unconstrained {\rm OBOP} solutions are
\[
{\rm 6 \mid 9 \mid 5 \mid 3\ 7 \mid 1 \mid\mid 2\ 4\ 8\ 10}
\qquad\text{and}\qquad
{\rm 6 \mid 9 \mid 5 \mid 3 \mid 7 \mid 1 \mid\mid 2\ 4\ 8\ 10},
\]
both with objective value \(29.1\), which is worse than the {\rm TCU-\(k\)} optimum.

Similarly, for \(k=3\), the {\rm TCU-\(k\) OBOP} yields
\[
{\rm 6 \mid 9 \mid 5 \mid\mid 1\ 2\ 3\ 4\ 7\ 8\ 10}
\]
with objective value \(30.1\). By contrast, the two-bucket variant with a first bucket of size \(k=3\) gives
\[
{\rm 6\ 7\ 9 \mid 1\ 2\ 3\ 4\ 5\ 8\ 10}
\]
with objective value \(31.38\). Moreover, visually splitting the unconstrained {\rm OBOP} ranking into two blocks of sizes \(3\) and \(7\) gives
\[
{\rm 5\ 6\ 9 \mid 1\ 2\ 3\ 4\ 7\ 8\ 10}
\]
with objective value \(32.06\). This illustrates that the {\rm TCU-\(k\)} formulation optimizes a different feasible region: it searches for the best weak order with a collapsed tail of size \(n-k\), rather than truncating or coarsening an independently computed ranking.
\end{example}

\begin{remark}
The previous example shows that changing the imposed bucket structure may substantially alter the optimal grouping. This phenomenon is not specific to weak-order aggregation. Similar behavior appears in non-hierarchical clustering methods, such as \(k\)-means, and in facility-location problems, such as the \(p\)-median problem, where changing the number or size of groups can lead to different optimal configurations; see \cite{laporte2019location} for a survey.
\end{remark}

\subsection{Incorporating fairness constraints} \label{sec:fairness}

Fairness is a central concern in ranking and recommendation systems, where it is often necessary to regulate the representation of different item groups by promoting some groups or limiting the dominance of others. Recent work in rank aggregation has introduced mechanisms to address protected attributes, enabling group-aware consensus rankings \citep{chakraborty2022fair,pitoura2022fairness}. For instance, an academic ranking may require a minimum proportion of students from underrepresented schools among top positions, or limit the share of any single school.

Building on \cite{chakraborty2025improved}, who define fairness in linear orderings by controlling group proportions in top ranks, we extend this idea to a bucket-based setting with ties. We incorporate fairness constraints into the weak-order framework using assignment variables \(\bm{y}\), as in model~\eqref{mod:p-bucket_1}, allowing up to \(n\) buckets. Empty buckets are placed at the end to preserve the order of filled ones.

Let \(G_1,\dots,G_g\subseteq[n]\) be predefined item groups, and let \(\lambda_{i\ell},\mu_{i\ell}\in[0,1]\) denote lower and upper bounds on the desired proportion of items from group \(G_i\) within the first \(\ell\) buckets. If no fairness requirement is imposed for a given pair \((i,\ell)\), we take the default values \(\lambda_{i\ell}=0\) and \(\mu_{i\ell}=1\). The corresponding integer constraints are
\begin{equation}\label{eq:fairness_integer}
\biggl\lfloor 
\lambda_{i\ell}\sum_{r=1}^n\sum_{u=1}^{\ell}y_{ru}
\biggr\rfloor
\le
\sum_{s\in G_i}\sum_{u=1}^{\ell}y_{su}
\le
\biggl\lceil 
\mu_{i\ell}\sum_{r=1}^n\sum_{u=1}^{\ell}y_{ru}
\biggr\rceil,
\qquad
\forall i\in[g],\ \forall \ell\in[n].
\end{equation}

The floor and ceiling operators are needed because proportional requirements must be interpreted in an integer setting. For example, requiring a \(50\%\) representation in a prefix containing three items should allow either one or two items from the group, rather than making the model infeasible. Moreover, the constraints are imposed cumulatively on the first \(\ell\) buckets, rather than bucket by bucket. This follows the usual approach in fair ranking \citep{wei2022rank,chakraborty2025improved} and avoids the weakness of per-bucket constraints, which may become almost nonbinding when buckets contain very few items.

The following example illustrates both issues. First, it shows why the floor and ceiling operators are needed to make proportional requirements meaningful in a discrete setting. Then, it explains why these requirements must be imposed cumulatively over prefixes, rather than independently on each bucket.

\begin{example}
Consider items $[6]$ and one protected group \(G_1=\{1,2,3\}\). Suppose that we require a \(50\%\) representation of \(G_1\) in the first two cumulative bucket prefixes, that is, \(\lambda_{1\ell}=\mu_{1\ell}=0.5\) for \(\ell\in\{1,2\}\). All remaining items form the complementary group, for which no bounds are imposed.

First, the floor and ceiling operators are necessary because the size of a cumulative prefix need not be even. For instance, consider the bucket order
\[
{\rm 1\ 4\ \big|\ 2\ \big|\ 3\ 5\ 6}.
\]
The first two buckets contain three items, namely \(\{1,2,4\}\), two of which belong to \(G_1\). Requiring an exact \(50\%\) proportion would demand \(1.5\) items from \(G_1\), which is not meaningful in an integer setting. With the rounded bounds, however, the requirement becomes
\[
\lfloor 0.5\cdot 3\rfloor = 1
\qquad\text{and}\qquad
\lceil 0.5\cdot 3\rceil = 2,
\]
so this bucket order is feasible.

Second, the bounds must be imposed cumulatively rather than independently on each bucket. Indeed, for a one-item bucket, the rounded \(50\%\) bounds are \(0\) and \(1\), so either group may occupy that bucket. Therefore, purely local constraints would allow a ranking such as
\[
{\rm 4\ \big|\ 5\ \big|\ 6\ \big|\ 1\ \big|\ 2\ \big|\ 3},
\]
even though group \(G_1\) is entirely absent from the upper part of the ranking. By contrast, cumulative bounds prevent this behaviour by controlling representation over prefixes of buckets. For example,
\[
{\rm 1\ \big|\ 4\ \big|\ 5\ \big|\ 2\ \big|\ 3\ \big|\ 6}
\]
satisfies the required \(50\%\) representation in the relevant cumulative prefixes.

Thus, rounding ensures that proportional fairness is meaningful for prefixes of odd cardinality, while cumulative constraints prevent protected groups from being systematically postponed or overrepresented in the upper part of the ranking.
\end{example}

The parameters \(\lambda\) and \(\mu\) must be chosen consistently. In particular, when the groups form a partition, natural feasibility requirements include \(\sum_{i\in[g]}\lambda_{i\ell}\le 1\) and \(\sum_{i\in[g]}\mu_{i\ell}\ge 1\) for all \(\ell\in[n]\). Moreover, at the final prefix, which contains all items, the bounds must be compatible with the actual group proportions; in particular,
\[
\lambda_{in}\le \frac{|G_i|}{n}\le \mu_{in}
\qquad \forall i\in[g].
\]
Thus, the parameters should be calibrated to reflect realistic group proportions while preserving feasibility.

To keep the formulation linear, we replace the floor and ceiling operators with equivalent linear inequalities. Write
\[
\lambda_{i\ell}=\frac{\eta_{i\ell}}{\rho_{i\ell}},
\qquad
\mu_{i\ell}=\frac{\kappa_{i\ell}}{\tau_{i\ell}},
\]
as irreducible fractions with positive denominators. The following result gives the required linearization.

\begin{lemma}\label{lem:linearization}
The integer fairness constraints~\eqref{eq:fairness_integer} are equivalent, under integer variables, to the following linear inequalities:
\begin{subequations}\label{eq:linearized_fairness}
\begin{align}
& \eta_{i\ell}\sum_{r=1}^n\sum_{u=1}^{\ell}y_{ru}
\le
\rho_{i\ell}\sum_{s\in G_i}\sum_{u=1}^{\ell}y_{su}
+(\rho_{i\ell}-1),
&& \forall i\in[g],\ \forall \ell\in[n],
\label{eq:linearized_fairness_lower}\\
& \tau_{i\ell}\sum_{s\in G_i}\sum_{u=1}^{\ell}y_{su}
-(\tau_{i\ell}-1)
\le
\kappa_{i\ell}\sum_{r=1}^n\sum_{u=1}^{\ell}y_{ru},
&& \forall i\in[g],\ \forall \ell\in[n].
\label{eq:linearized_fairness_upper}
\end{align}
\end{subequations}
\end{lemma}

\begin{proof}
Fix \(i\in[g]\) and \(\ell\in[n]\), and set
\[
T=\sum_{r=1}^n\sum_{u=1}^{\ell}y_{ru},
\qquad
S=\sum_{s\in G_i}\sum_{u=1}^{\ell}y_{su}.
\]
Both \(T\) and \(S\) are integer. For the lower bound,
\[
\lfloor \lambda_{i\ell}T\rfloor \le S
\Longleftrightarrow
\lambda_{i\ell}T<S+1
\Longleftrightarrow
\eta_{i\ell}T<\rho_{i\ell}S+\rho_{i\ell}
\Longleftrightarrow
\eta_{i\ell}T\le \rho_{i\ell}S+(\rho_{i\ell}-1),
\]
where the last equivalence follows from integrality. Similarly, for the upper bound,
\[
S\le \lceil \mu_{i\ell}T\rceil
\Longleftrightarrow
S<\mu_{i\ell}T+1
\Longleftrightarrow
\tau_{i\ell}S<\kappa_{i\ell}T+\tau_{i\ell}
\Longleftrightarrow
\tau_{i\ell}S-(\tau_{i\ell}-1)\le \kappa_{i\ell}T.
\]
This proves the equivalence.
\end{proof}

We now integrate these fairness conditions into the weak-order framework. Since the number of nonempty buckets is not fixed in advance, we use \(n\) potential buckets and force empty buckets, if any, to appear only at the end of the ranking:
\begin{subequations}\label{mod:fairness}
    \begin{align}
    \min\quad & f(\Pi,\bm{x})\\
    \text{s.t.}\quad
        & \sum_{u=1}^n y_{ru}=1
        && \forall r\in[n], \label{eq:elemento_en_bucket_8}\\
        & \sum_{s=1}^n y_{su}\ge \sum_{v=u+1}^n y_{rv}
        && \forall r,u\in[n]: u<n, \label{eq:buckets_llenos_en_orden_8}\\
        & y_{ru}+y_{su}\le x_{rs}+x_{sr}
        && \forall r,s,u\in[n]: r<s, \label{eq:empates_bucket_8}\\
        & x_{sr}+\sum_{v=1}^{u}y_{rv}
        +\sum_{v=u+1}^{n}y_{sv}\le 2
        && \forall r,s,u\in[n]: r\neq s,\ u<n, \label{eq:buckets_ordenados_8}\\
        & \eqref{eq:linearized_fairness_lower},\ \eqref{eq:linearized_fairness_upper}\\
        & x_{rs}\in\{0,1\}
        && \forall r,s\in[n]: r\neq s, \label{eq:naturaleza_x_8}\\
        & y_{ru}\in\{0,1\}
        && \forall r,u\in[n]. \label{eq:naturaleza_y_8}
    \end{align}
\end{subequations}

Constraints~\eqref{eq:elemento_en_bucket_8}, \eqref{eq:empates_bucket_8}, and~\eqref{eq:buckets_ordenados_8} play the same role as in the assignment-based formulation. Constraint~\eqref{eq:buckets_llenos_en_orden_8} ensures that empty buckets appear only after all nonempty buckets. Constraints~\eqref{eq:linearized_fairness_lower}--\eqref{eq:linearized_fairness_upper} impose the prescribed group proportions on cumulative prefixes of the bucket order.

Specific choices of \(\lambda\) and \(\mu\) recover existing fairness notions. For instance, in \cite{wei2022rank}, a ranking is deemed fair when both bounds are fixed to the global proportion of each group, that is,
\[
\lambda_{i\ell}=\mu_{i\ell}=\frac{|G_i|}{n}
\qquad \forall i\in[g],\ \forall \ell\in[n],
\]
with groups forming a disjoint partition of the item set.

After illustrating the role of rounded cumulative proportional constraints, we return to the OBOP running example.

\begin{example}\label{ex:ejemplo_8}
Continuing with Example~\ref{ex:ejemplo_1}, consider the groups
\[
G_1=\{1,3,4,8\},
\qquad
G_2=\{2,5,6,7\}.
\]
Assume that \(\lambda_{i\ell}=\mu_{i\ell}=|G_i|/n\) for all \(i\in[g]\) and \(\ell\in[n]\). The {\rm Fair OBOP}, obtaineud by adapting model~\eqref{mod:fairness} to the {\rm OBOP}, yields the unique optimum
\[
{\rm 3 \mid 1\ 2\ 4\ 7 \mid 5\ 8 \mid 6}
\]
with objective value \(11.86\). The unconstrained {\rm OBOP} optimum from Example~\ref{ex:ejemplo_1} is not feasible under these fairness constraints: its first bucket contains two items, both from \(G_1\), while the imposed proportional bounds require a balanced representation between the two groups.
\end{example}

More generally, fairness requirements can also be expressed through explicit lower and upper bounds. For example, one may impose
\[
L_{iu}\le \sum_{r\in G_i}y_{ru}\le U_{iu}
\qquad
\text{or}
\qquad
L_{i\ell}\le \sum_{r\in G_i}\sum_{u=1}^{\ell}y_{ru}\le U_{i\ell},
\]
depending on whether the bounds refer to a specific bucket \(u\) or to the first \(\ell\) buckets.

The fairness model can also be combined with the structural variants introduced above. If the number of buckets is fixed to \(p\), the index set for \(u\) is replaced by \([p]\), the nonemptiness constraints~\eqref{eq:buckets_no_vacios_3} are imposed, and constraint~\eqref{eq:buckets_llenos_en_orden_8} is no longer needed. If prescribed bucket sizes \(q_u\) are imposed through~\eqref{eq:prescribed_bucket_sizes_assignment}, then
\[
\sum_{r=1}^n\sum_{u=1}^{\ell}y_{ru}
=
\sum_{u=1}^{\ell}q_u,
\]
so the total size of each prefix is fixed in advance and the floor and ceiling terms can be handled directly. Similarly, fairness can be incorporated into the Tail-Collapsed Upper-\(k\) formulation by imposing lower and upper bounds on the composition of the collapsed tail:
\[
L_i\le \sum_{r\in G_i}z_r\le U_i,
\qquad \forall i\in[g].
\]

\section{Computational experiments} \label{sec:CE}

We now present the computational experiments conducted to evaluate the proposed exact formulation for the {\rm OBOP} and the constrained variants introduced in the previous sections. The experiments are divided into two parts. First, we compare our exact formulation for the {\rm OBOP} with the heuristic methods proposed in \cite{aledo2018approaching}. This comparison allows us to determine which of the previously reported heuristic solutions are in fact optimal and which can be improved. Second, we solve selected instances from PrefLib under different structural and fairness constraints, illustrating how the proposed framework can incorporate additional requirements.

The instances used in this study are derived from data obtained from the PrefLib\footnote{\url{http://www.preflib.org/}} library and the MovieLens\footnote{\url{https://grouplens.org/datasets/movielens/}} dataset, and can be downloaded from the following repository.\footnote{\url{https://doi.org/10.5281/zenodo.18799994}.} All tests were performed using the commercial IP solver \textit{Xpress Mosel} on a server equipped with an AMD Ryzen~9~7950X processor (3.4~GHz, 16~cores and 32~threads) and 4\,$\times$\,48~GB of DDR5 RAM (196~GB in total). In our experiments, each Mosel instance was assigned 4 computational threads, and 4 jobs were executed in parallel. This configuration balanced solver performance with controlled resource usage across the benchmark set.

\subsection{Comparison with heuristic methods for {\rm OBOP}}

We compare the exact {\rm OBOP} formulation introduced in Section~\ref{Sec:OBOP} with the heuristic methods proposed in \cite{aledo2018approaching}. The purpose of this comparison is not to perform a running-time benchmark against those heuristics, since the two approaches have different goals and experimental protocols. Heuristic methods are designed to obtain good solutions quickly, whereas our exact formulation is designed to certify optimality. Therefore, the comparison focuses on solution quality: we determine which of the previously reported heuristic solutions are in fact optimal and which can be improved.

The benchmark set considered in this comparison consists of instances from PrefLib and MovieLens, coinciding with those used in \cite{aledo2018approaching}, thus allowing a direct comparison with the best heuristic solution previously reported for each instance. Each instance is defined by the number of items \(n\) and the number of voters \(m\), which correspond to the input rankings used to construct the pair order matrix \(C\).

The exact formulation contains a large number of transitivity constraints. We therefore consider two exact implementations. The first one is a static formulation, in which all transitivity constraints are added a priori. The second one is a branch-and-cut implementation based on lazy constraints. In this implementation, the initial model includes only the transitivity inequalities associated with triples of distinct items \(r,s,t\) satisfying \(c_{rs} \geq 0.25\), \(c_{st} \geq 0.25\), and \(c_{tr} \geq 0.25\). Any remaining transitivity inequalities that are violated by an incumbent solution are dynamically added during the branch-and-bound process. Both implementations solve the same exact formulation and certify optimality whenever they terminate.

Table~\ref{tab:results_heu_vs_exact} reports the results. For each instance, we provide the number of items and voters, the best objective value obtained by the heuristics in \cite{aledo2018approaching}, and the optimal value certified by our exact formulation. We also report the value of the linear relaxation and the running times of the static and lazy-constraint implementations. These times are included to assess the computational behavior of the exact approach, i.e. the model,  and to compare the two exact implementations, not to benchmark against the heuristic procedures.

\begingroup
\scriptsize
\setlength{\LTleft}{\fill}
\setlength{\LTright}{\fill}
\setlength{\tabcolsep}{4pt}

\begin{longtable}[c]{l r r r r r r}
\caption{Comparison between the best heuristic solutions reported in \cite{aledo2018approaching} and the optimal values certified by the exact {\rm OBOP} formulation. The columns ``Time'' and ``Time Lazy'' report the running times, in seconds, of the static and lazy-constraint implementations, respectively. Bold values indicate instances where the exact formulation improves the best heuristic solution. All numerical values are rounded to two decimals.}
\label{tab:results_heu_vs_exact}\\
\toprule
Instance & Items & Voters & Best Heu. & \multicolumn{3}{c}{Exact} \\
\cmidrule(lr){4-4}\cmidrule(lr){5-7}
& & & Obj.\ val. & Obj.\ val. & Time & Time Lazy \\
\midrule
\endfirsthead

\caption[]{Comparison between heuristic and exact approaches for the {\rm OBOP} (continued).}\\
\toprule
Instance & Items & Voters & Best Heu. & \multicolumn{3}{c}{Exact} \\
\cmidrule(lr){4-4}\cmidrule(lr){5-7}
& & & Obj.\ value & Obj.\ value & Time & Time Lazy \\
\midrule
\endhead

\midrule
\multicolumn{7}{r}{\scriptsize\itshape (continued on next page)}\\
\endfoot

\bottomrule
\endlastfoot

ED-10-02 & 51 & 9 & 524.97 & 524.97 & 7.26 & 2.20 \\
ED-10-03 & 54 & 10 & 513.38 & 513.38 & 6.57 & 2.86 \\
ED-10-11 & 50 & 11 & 558.01 & 558.01 & 6.28 & 2.45 \\
ED-10-14 & 62 & 15 & 744.64 & 744.64 & 10.60 & 8.41 \\
ED-10-15 & 52 & 14 & 576.96 & 576.96 & 6.64 & 3.02 \\
ED-10-16 & 57 & 16 & 585.07 & 585.07 & 9.63 & 3.29 \\
ED-10-17 & 61 & 17 & 683.76 & 683.76 & 12.12 & 3.39 \\
ED-11-01 & 240 & 5 & 5560.00 & \textbf{5556.20} & 39801.98 & 184.48 \\
ED-11-02 & 242 & 5 & 10913.40 & 10913.40 & 77452.66 & 1971.17 \\
ED-11-03 & 103 & 5 & 1653.00 & 1653.00 & 229.51 & 4.77 \\
ED-14-02 & 100 & 5000 & 1953.24 & 1953.24 & 1017.29 & 92.76 \\
ED-15-07 & 110 & 4 & 1443.00 & 1443.00 & 565.47 & 54.18 \\
ED-15-08 & 99 & 4 & 1304.50 & 1304.50 & 367.55 & 63.56 \\
ED-15-09 & 115 & 4 & 1684.00 & 1684.00 & 536.54 & 98.32 \\
ED-15-10 & 71 & 4 & 666.50 & 666.50 & 101.04 & 20.50 \\
ED-15-11 & 63 & 4 & 484.00 & 484.00 & 34.35 & 14.06 \\
ED-15-13 & 93 & 4 & 1381.50 & 1381.50 & 211.81 & 93.43 \\
ED-15-14 & 163 & 4 & 3505.50 & \textbf{3502.50} & 2945.57 & 345.06 \\
ED-15-15 & 69 & 4 & 728.00 & 728.00 & 107.15 & 8.54 \\
ED-15-16 & 70 & 4 & 673.00 & 673.00 & 88.73 & 8.22 \\
ED-15-17 & 127 & 4 & 2304.00 & 2304.00 & 848.89 & 223.13 \\
ED-15-19 & 87 & 4 & 999.50 & 999.50 & 271.52 & 42.18 \\
ED-15-20 & 122 & 4 & 2382.50 & 2382.50 & 794.26 & 285.38 \\
ED-15-21 & 96 & 4 & 1246.50 & 1246.50 & 271.38 & 57.23 \\
ED-15-22 & 112 & 4 & 1720.50 & 1720.50 & 530.64 & 148.31 \\
ED-15-23 & 142 & 4 & 2840.00 & \textbf{2838.00} & 6148.12 & 332.74 \\
ED-15-24 & 91 & 4 & 1028.50 & 1028.50 & 147.32 & 1028.50 \\
ED-15-26 & 82 & 4 & 844.00 & 844.00 & 42.49 & 47.66 \\
ED-15-27 & 95 & 4 & 1220.50 & 1220.50 & 61.23 & 137.82 \\
ED-15-28 & 102 & 4 & 1518.50 & 1518.50 & 125.28 & 104.05 \\
ED-15-29 & 106 & 4 & 1410.00 & 1410.00 & 85.32 & 84.86 \\
ED-15-30 & 64 & 4 & 560.50 & 560.50 & 17.88 & 20.63 \\
ED-15-31 & 67 & 4 & 589.50 & 589.50 & 14.25 & 589.50 \\
ED-15-32 & 153 & 4 & 3112.50 & 3112.50 & 604.38 & 561.76 \\
ED-15-33 & 128 & 4 & 2298.50 & 2298.50 & 230.87 & 294.50 \\
ED-15-34 & 55 & 4 & 444.50 & 444.50 & 4.30 & 16.68 \\
ED-15-35 & 68 & 4 & 674.50 & 674.50 & 15.37 & 31.76 \\
ED-15-39 & 89 & 4 & 922.00 & 922.00 & 43.33 & 38.48 \\
ED-15-40 & 131 & 4 & 2512.00 & 2512.00 & 495.29 & 398.44 \\
ED-15-51 & 77 & 4 & 777.00 & 777.00 & 105.69 & 117.42 \\
ED-15-54 & 60 & 4 & 410.50 & 410.50 & 5.84 & 3.44 \\
ED-15-57 & 73 & 4 & 821.50 & 821.50 & 22.73 & 13.59 \\
ED-15-60 & 72 & 4 & 658.50 & 658.50 & 28.99 & 658.50 \\
ED-15-69 & 81 & 4 & 809.00 & 809.00 & 30.53 & 14.97 \\
ED-15-77 & 56 & 4 & 498.67 & 498.67 & 13.36 & 3.91 \\
ML-050-6 & 50 & 936 & 418.21 & 418.21 & 4.57 & 1.90 \\
ML-100-1 & 100 & 951 & 1869.63 & 1869.63 & 166.02 & 53.64 \\
ML-200-3 & 200 & 936 & 8421.68 & 8421.68 & 4079.50 & 14964.12 \\
ML-250-4 & 250 & 942 & 12281.83 & 12281.83 & 104304.45 & 50530.62 \\
ML-250-5 & 250 & 939 & 12257.52 & 12257.52 & 121672.11 & 12227.59 \\

\end{longtable}
\endgroup

\begin{remark}
The heuristic objective values reported in Table~\ref{tab:results_heu_vs_exact} were directly extracted from the best values published in \cite{aledo2018approaching}. In that study, nine evolutionary algorithms were evaluated through 30 independent runs for each instance, amounting to 270 runs per instance. Concerning the computational effort involved, the authors report running times ranging from less than one second per run for small instances to approximately \(2\) minutes per run for the largest instances. These times are mentioned only to contextualize the effort underlying the published heuristic benchmarks; no direct running-time comparison with our exact approach is intended, since the two studies were conducted under different implementations, hardware platforms, and experimental objectives.
\end{remark}

Table~\ref{tab:results_heu_vs_exact} highlights the main contribution of the exact formulation: it transforms previously heuristic evidence into certified optimality results. In most instances, the best heuristic value reported in \cite{aledo2018approaching} coincides with the optimal value obtained by our formulation. Thus, the proposed optimization model enables us to prove that these heuristic solutions were, in many cases, globally optimal bucket orders.

At the same time, the exact formulation identifies instances in which the best heuristic solution was still suboptimal. In particular, for ED-11-01, ED-15-14, and ED-15-23, our model improves the best previously reported objective values. These cases illustrate the practical value of having an exact formulation: it can both certify optimality when a heuristic solution is optimal and detect improvements when it is not.

The instance ED-10-50 was excluded from the table because it could not be solved to proven optimality within the imposed time limit of two days, and the remaining optimality gap was still too large to provide a meaningful comparison. Here, the optimality gap is the relative difference between the best feasible objective value found by the solver and its best lower bound at termination.

Regarding running times, the results should be interpreted only within the exact approach. The lazy-constraint implementation is particularly effective in several large instances, notably ED-11-01, ED-11-02, ED-15-14, and ED-15-23, where it substantially reduces the solution time with respect to the static formulation. However, it does not uniformly dominate the static formulation, which remains competitive in some instances. Overall, both implementations are useful: the static model provides a direct benchmark when manageable, while the lazy-constraint version is especially valuable for larger or more difficult instances. In all cases, both implementations preserve the exact nature of the approach.

\subsection{Testing our models for the {\rm OBOP} variants}
\label{subsec:testing_obop_variants}

In this section, we report the computational experiments carried out to evaluate the behaviour of the proposed models when applied to the {\rm OBOP} and to several constrained variants. The aim of these experiments is not only to assess the computational performance of the formulations, but also to analyse how the optimal consensus bucket order changes when additional structural or fairness requirements are imposed. In all constrained variants considered below, the objective function remains the {\rm OBOP} distance \(D(B,C)\); only the feasible region changes through the additional constraints. Therefore, increases in the objective value can be interpreted as the cost of imposing the corresponding structural or fairness requirements.

We consider a collection of benchmark instances taken from the PrefLib\footnote{\url{http://www.preflib.org/}} library. Specifically, we select instances from three datasets: \textit{Poland Local Elections} (LPE-68-xx), \textit{Boxing} (BOX-42-xx), and \textit{Seasons Power Ranking} (SPR-56-xx). These datasets provide a diverse collection of preference profiles, with different numbers of items and voters, and therefore allow us to test the proposed formulations under different aggregation scenarios.

Each instance consists of a set of \(n\) items, a number \(m\) of voters, and a collection of individual rankings. The rankings are strict, but they are not necessarily complete. In order to adapt these data to the {\rm OBOP} framework, we first construct a pairwise comparison matrix \(A=(a_{rs})_{n\times n}\), where \(a_{rs}\) denotes the number of voters who rank item \(r\) ahead of item \(s\). Based on this matrix, we define the pair order matrix \(C=(c_{rs})_{n\times n}\) as follows:
\[
c_{rs} =
\begin{cases}
    \dfrac{a_{rs}}{a_{rs}+a_{sr}}, & \text{if } a_{rs}+a_{sr} \neq 0, \\[6pt]
    0.5, & \text{if } a_{rs}+a_{sr} = 0.
\end{cases}
\]
Thus, \(c_{rs}\) represents the empirical proportion of observed comparisons in which item \(r\) is preferred to item \(s\). When no voter compares the pair \((r,s)\), we set \(c_{rs}=0.5\), reflecting the absence of evidence in favour of either ordering.

All numerical values reported in the tables are rounded to two decimal places. The experiments were designed to compare the standard {\rm OBOP} formulation with three extensions: the fixed-number-of-buckets model, the Tail-Collapsed Upper-\(k\) model, and the fairness-constrained model.

\subsubsection{{\rm OBOP}}

We first solve the selected instances using the standard {\rm OBOP} formulation~\eqref{mod:OBOP}. These optimal values will be used as a reference to evaluate the constrained variants considered below. Since the subsequent comparisons focus on changes in the objective value, Table~\ref{tab:OBOP} reports only the optimal value, the linear relaxation bound, and basic information on the structure of the optimal solutions.

\begingroup
\scriptsize
\setlength{\tabcolsep}{3pt}
\begin{longtable}{lrrrrc}
\caption{Results obtained for the standard {\rm OBOP} model. The columns report the instance name, the number of items \(n\), the number of voters \(m\), the optimal objective value, the number of optimal solutions, and the number of buckets in each optimal solution.}
\label{tab:OBOP}\\
\toprule
Instance & \(n\) & \(m\) & Obj. value & \#Optima & \#Buckets \\
\midrule
\endfirsthead

\caption[]{Results obtained for the standard {\rm OBOP} model \textit{(continued)}.}\\
\toprule
Instance & \(n\) & \(m\) & Obj. value & \#Optima & \#Buckets \\
\midrule
\endhead

\midrule
\multicolumn{6}{r}{\textit{Continued on next page}}\\
\endfoot

\bottomrule
\endlastfoot

LPE-68-01 & 10 & 23 & 15.63 & 1 & 4 \\
LPE-68-02 & 14 & 15 & 34.14 & 2 & 3,4 \\
LPE-68-03 & 9 & 15 & 12.29 & 1 & 4 \\
LPE-68-04 & 17 & 15 & 29.78 & 1 & 3 \\
LPE-68-05 & 5 & 15 & 3.94 & 1 & 1 \\
LPE-68-06 & 9 & 15 & 11.25 & 1 & 3 \\
LPE-68-07 & 13 & 15 & 25.73 & 1 & 3 \\
LPE-68-08 & 7 & 15 & 6.38 & 1 & 4 \\
LPE-68-09 & 10 & 21 & 6.99 & 2 & 7,8 \\
LPE-68-10 & 7 & 15 & 4.97 & 1 & 2 \\[0.2em]

BOX-42-01 & 17 & 31 & 14.27 & 3 & 14,15,15 \\
BOX-42-02 & 23 & 46 & 33.70 & 1 & 9 \\
BOX-42-03 & 20 & 36 & 22.94 & 2 & 11,12 \\
BOX-42-04 & 21 & 36 & 32.69 & 1 & 9 \\
BOX-42-05 & 22 & 37 & 27.18 & 2 & 13,14 \\
BOX-42-06 & 22 & 37 & 17.38 & 1 & 14 \\
BOX-42-07 & 25 & 41 & 51.58 & 2 & 10,11 \\
BOX-42-08 & 22 & 38 & 28.01 & 1 & 12 \\
BOX-42-09 & 18 & 9 & 3.89 & 1 & 13 \\
BOX-42-10 & 16 & 31 & 1.45 & 1 & 13 \\[0.2em]

SPR-56-01 & 31 & 5 & 76.07 & 1 & 13 \\
SPR-56-03 & 41 & 5 & 155.97 & 1 & 10 \\
SPR-56-08 & 35 & 5 & 83.13 & 1 & 9 \\
SPR-56-12 & 31 & 5 & 73.30 & 2 & 11,12 \\
SPR-56-15 & 29 & 5 & 62.17 & 1 & 16 \\
SPR-56-16 & 42 & 15 & 252.98 & 1 & 7 \\
SPR-56-19 & 55 & 15 & 402.95 & 1 & 5 \\
SPR-56-29 & 40 & 15 & 180.51 & 1 & 7 \\
SPR-56-30 & 43 & 14 & 266.56 & 1 & 6 \\
SPR-56-94 & 47 & 17 & 305.39 & 1 & 4 \\

\end{longtable}
\endgroup

All instances were solved to optimality in less than five seconds, so computation times are omitted from the table. The results confirm that the standard formulation is computationally tractable for these benchmark sizes and provides a suitable baseline for the constrained variants.

Most instances have a unique optimal solution, although multiple optima appear in some cases. For the selected instances, the numbers of buckets attained by optimal solutions form a set of consecutive values.

\subsubsection{\texorpdfstring{$p$}{p}-{\rm OBOP}}

We compare the standard {\rm OBOP} with variants in which the bucket structure is partially prescribed. We consider fixed numbers of buckets \(p=2\), \(p=5\), \(p=\underline p-1\), and \(p=\overline p+1\), where \(\underline p\) and \(\overline p\) are, respectively, the minimum and maximum number of buckets attained by an optimal solution of the standard {\rm OBOP}. We also consider a prescribed-size variant with \(p=3\), where \(q_1=\lceil n/10\rceil\), \(q_2=\lceil n/5\rceil\), and \(q_3=n-q_1-q_2\). All these variants are solved using the assignment-based formulation given by model~\eqref{mod:p-bucket_1}, reinforced with inequalities~\eqref{eq:antes_o_despues_1} and~\eqref{eq:transitividad_1}.

\begingroup
\scriptsize
\setlength{\tabcolsep}{3pt}
\begin{longtable}{lrr rrrrrr}
\caption{Results obtained for the extended \(p\)-{\rm OBOP} model. The columns report the instance name, the number of items \(n\), the number of voters \(m\), the optimal value of the standard {\rm OBOP}, and the optimal values of the \(p\)-{\rm OBOP} variants. Here, \(\underline p\) and \(\overline p\) denote, respectively, the smallest and largest numbers of buckets yielding an optimal {\rm OBOP} solution. The last column corresponds to the prescribed-size model with \(p=3\), \(q_1=\lceil n/10\rceil\), \(q_2=\lceil n/5\rceil\), and the remaining items assigned to the last bucket.}
\label{tab:p-OBOP}\\
\toprule
Instance & \(n\) & \(m\) & Opt. OBOP & \multicolumn{5}{c}{\(p\)-{\rm OBOP}} \\
\cmidrule(lr){5-9}
 & & & & \(p=2\) & \(p=5\) & \(p=\underline p-1\) & \(p=\overline p+1\) & \shortstack{\(p=3\)\\ \(q_1=\lceil n/10\rceil\)\\ \(q_2=\lceil n/5\rceil\)} \\
\midrule
\endfirsthead

\caption[]{Results obtained for the extended \(p\)-{\rm OBOP} model \textit{(continued)}.}\\
\toprule
Instance & \(n\) & \(m\) & Opt. OBOP & \multicolumn{5}{c}{\(p\)-{\rm OBOP}} \\
\cmidrule(lr){5-9}
 & & & & \(p=2\) & \(p=5\) & \(p=\underline p-1\) & \(p=\overline p+1\) & \shortstack{\(p=3\)\\ \(q_1=\lceil n/10\rceil\)\\ \(q_2=\lceil n/5\rceil\)} \\
\midrule
\endhead

\midrule
\multicolumn{9}{r}{\textit{Continued on next page}}\\
\endfoot

\bottomrule
\endlastfoot

LPE-68-01 & 10 & 23 & 15.63 & 17.07 & 15.68 & 15.92 & 15.68 & 17.13 \\
LPE-68-02 & 14 & 15 & 34.14 & 35.14 & 36.14 & 35.14 & 36.14 & 42.14 \\
LPE-68-03 & 9 & 15 & 12.29 & 13.29 & 13.29 & 13.29 & 13.29 & 13.29 \\
LPE-68-04 & 17 & 15 & 29.78 & 35.67 & 37.44 & 35.67 & 32.44 & 56.11 \\
LPE-68-05 & 5 & 15 & 3.94 & 4.01 & 6.06 & -- & 4.01 & 5.01 \\
LPE-68-06 & 9 & 15 & 11.25 & 11.75 & 14.25 & 11.75 & 13.25 & 16.25 \\
LPE-68-07 & 13 & 15 & 25.73 & 31.26 & 26.71 & 31.26 & 26.26 & 25.73 \\
LPE-68-08 & 7 & 15 & 6.38 & 6.73 & 7.33 & 6.38 & 7.33 & 8.93 \\
LPE-68-09 & 10 & 21 & 6.99 & 16.71 & 8.14 & 7.14 & 7.54 & 18.49 \\
LPE-68-10 & 7 & 15 & 4.97 & 4.97 & 8.60 & 9.40 & 5.77 & 5.77 \\[0.2em]

BOX-42-01 & 17 & 31 & 14.27 & 56.12 & 27.85 & 15.01 & 15.17 & 54.12 \\
BOX-42-02 & 23 & 46 & 33.70 & 87.50 & 37.39 & 33.83 & 34.57 & 77.90 \\
BOX-42-03 & 20 & 36 & 22.94 & 76.07 & 32.98 & 22.98 & 23.40 & 73.22 \\
BOX-42-04 & 21 & 36 & 32.69 & 82.99 & 40.20 & 33.10 & 32.95 & 73.28 \\
BOX-42-05 & 22 & 37 & 27.18 & 85.95 & 35.00 & 27.42 & 27.57 & 74.74 \\
BOX-42-06 & 22 & 37 & 17.38 & 89.81 & 34.50 & 17.62 & 17.78 & 80.71 \\
BOX-42-07 & 25 & 41 & 51.58 & 123.71 & 64.17 & 51.71 & 52.16 & 132.32 \\
BOX-42-08 & 22 & 38 & 28.01 & 95.75 & 41.00 & 28.21 & 28.06 & 96.12 \\
BOX-42-09 & 18 & 9 & 3.89 & 64.33 & 18.33 & 4.00 & 4.22 & 67.11 \\
BOX-42-10 & 16 & 31 & 1.45 & 52.64 & 15.90 & 2.29 & 1.87 & 48.90 \\[0.2em]

SPR-56-01 & 31 & 5 & 76.07 & 172.07 & 83.87 & 76.27 & 76.67 & 163.27 \\
SPR-56-03 & 41 & 5 & 155.97 & 293.30 & 173.23 & 156.97 & 156.10 & 288.63 \\
SPR-56-08 & 35 & 5 & 83.13 & 200.53 & 92.73 & 83.33 & 83.53 & 242.93 \\
SPR-56-12 & 31 & 5 & 73.30 & 181.70 & 87.17 & 73.70 & 73.37 & 172.50 \\
SPR-56-15 & 29 & 5 & 62.17 & 140.63 & 77.37 & 62.30 & 62.50 & 175.03 \\
SPR-56-16 & 42 & 15 & 252.98 & 304.23 & 254.99 & 253.39 & 253.70 & 277.53 \\
SPR-56-19 & 55 & 15 & 402.95 & 486.56 & 402.95 & 405.35 & 403.42 & 463.24 \\
SPR-56-29 & 40 & 15 & 180.51 & 265.43 & 182.01 & 180.76 & 180.85 & 248.48 \\
SPR-56-30 & 43 & 14 & 266.56 & 306.13 & 267.87 & 267.87 & 266.90 & 316.11 \\
SPR-56-94 & 47 & 17 & 305.39 & 366.17 & 306.41 & 313.16 & 306.41 & 356.20 \\

\end{longtable}
\endgroup

The results in Table~\ref{tab:p-OBOP} show that the smallest deviations from the standard {\rm OBOP} optimum usually occur when the prescribed number of buckets is close to that of an optimal unconstrained solution, namely for \(p=\underline p-1\) or \(p=\overline p+1\). Larger deviations appear when \(p\) is far from this structure, especially for \(p=2\), where the ranking is forced to collapse into only two levels. The prescribed-size case is also restrictive, since it fixes not only the number of buckets but also the size of the first two levels.

Regarding computation times, all runs for \(p=2\), \(p=\underline p-1\), and \(p=\overline p+1\) were solved in less than 30 seconds. For \(p=5\), all but one instance were solved below that threshold, with a maximum time of 50.73 seconds. The prescribed-size variant was the most demanding: three instances exceeded 30 seconds, and the maximum time was 447.29 seconds. Thus, fixing only the number of buckets remains computationally mild in these experiments, whereas prescribing bucket sizes may noticeably increase the difficulty of the model.

For a more detailed view, Figure~\ref{fig:p-obop} reports the optimal value for every possible number of buckets on instances BOX-42-01 to BOX-42-05.

\begin{figure}
\centering
\caption{Optimal values of the $p$-OBOP for instances BOX-42-01 to BOX-42-05}
\label{fig:p-obop}

\begin{tikzpicture}

\begin{axis}[
    width=0.98\linewidth,
    height=0.62\linewidth,
    title={$p$-OBOP},
    xlabel={$p$},
    ylabel={Optimal value},
    xmin=0.5, xmax=23.5,
    ymin=10, ymax=92,
    xtick={1,2,...,23},
    minor xtick={2,4,6,8,10,12,14,16,18,20,22},
    ytick={10,20,30,40,50,60,70,80,90},
    minor y tick num=1,
    grid=major,
    minor grid style={gray!15},
    grid style={gray!25},
    tick label style={font=\scriptsize},
    label style={font=\small},
    title style={font=\small},
    axis x line*=bottom,
    axis y line*=left,
    clip=true
]

\pgfplotsset{
    piruleta/.style={
        ycomb,
        solid,
        line width=0.65pt,
        mark=*,
        mark size=1.5pt,
        mark options={solid}
    }
}

% BOX-42-01
\addplot[
    piruleta,
    color=boxblue,
    mark=*
]
table[
    x expr=\thisrow{p}-0.2,
    y=val
] {
p val
1 125.3234185
2 56.12341852
3 42.85069124
4 31.85069124
5 27.85069124
6 23.85069124
7 20.85069124
8 19.85069124
9 18.85069124
10 17.85069124
11 16.85069124
12 15.85069124
13 15.01198157
14 14.27004608
15 14.27004608
16 15.17480799
17 16.11029186
};

% BOX-42-02
\addplot[
    piruleta,
    color=boxred,
    mark=square*
]
table[
    x expr=\thisrow{p}-0.1,
    y=val
] {
p val
1 196.3997525
2 87.50425696
3 49.41516347
4 39.73447409
5 37.386648
6 35.386648
7 34.48619551
8 33.8340216
9 33.70358681
10 34.57315203
11 35.39923899
12 36.95862229
13 38.55527609
14 40.07306265
15 42.07306265
16 43.16947831
17 44.69115664
18 46.69115664
19 49.69115664
20 51.69115664
21 53.69115664
22 55.69115664
23 56.69115664
};

% BOX-42-03
\addplot[
    piruleta,
    color=boxgreen,
    mark=triangle*
]
table[
    x expr=\thisrow{p},
    y=val
] {
p val
1 161.5940637
2 76.07184146
3 46.07184146
4 38.69850812
5 32.98422241
6 30.85088908
7 28.85088908
8 26.65088908
9 24.65088908
10 22.98422241
11 22.94422241
12 22.94422241
13 23.39583531
14 23.72916864
15 24.1837141
16 24.79482521
17 25.57260299
18 26.90593632
19 27.79482521
20 28.57260299
};

% BOX-42-04
\addplot[
    piruleta,
    color=boxorange,
    mark=diamond*
]
table[
    x expr=\thisrow{p}+0.1,
    y=val
] {
p val
1 167.7626997
2 82.98789283
3 57.92533926
4 45.27892685
5 40.20346897
6 35.37021943
7 33.98560405
8 33.09671516
9 32.68930775
10 32.94645061
11 33.25364015
12 33.88109113
13 34.52161401
14 35.35494734
15 36.20679919
16 36.96605845
17 37.74383623
18 38.57716956
19 39.57716956
20 41.57716956
21 42.57716956
};

% BOX-42-05
\addplot[
    piruleta,
    color=boxpurple,
    mark=pentagon*
]
table[
    x expr=\thisrow{p}+0.2,
    y=val
] {
p val
1 198.1652926
2 85.94790132
3 57.3795734
4 45.51118327
5 35.00380194
6 32.10133425
7 30.22001815
8 28.97677491
9 28.27407221
10 27.92431068
11 27.68106744
12 27.42053361
13 27.17729037
14 27.17729037
15 27.56859472
16 28.21565354
17 29.16159949
18 30.16159949
19 31.20507775
20 32.20507775
21 33.16159949
22 34.16159949
};

% =========================================================
% Mínimos ligeramente más visibles
% =========================================================

% BOX-42-01: mínimo en p = 14 y p = 15
\addplot[
    only marks,
    forget plot,
    mark=*,
    mark size=2.0pt,
    mark options={
        solid,
        fill=boxblue!65!black,
        draw=boxblue!65!black
    }
]
coordinates {
    (13.8,14.27004608)
    (14.8,14.27004608)
};

% BOX-42-02: mínimo en p = 9
\addplot[
    only marks,
    forget plot,
    mark=square*,
    mark size=2.0pt,
    mark options={
        solid,
        fill=boxred!65!black,
        draw=boxred!65!black
    }
]
coordinates {
    (8.9,33.70358681)
};

% BOX-42-03: mínimo en p = 11 y p = 12
\addplot[
    only marks,
    forget plot,
    mark=triangle*,
    mark size=2.1pt,
    mark options={
        solid,
        fill=boxgreen!65!black,
        draw=boxgreen!65!black
    }
]
coordinates {
    (11,22.94422241)
    (12,22.94422241)
};

% BOX-42-04: mínimo en p = 9
\addplot[
    only marks,
    forget plot,
    mark=diamond*,
    mark size=2.1pt,
    mark options={
        solid,
        fill=boxorange!65!black,
        draw=boxorange!65!black
    }
]
coordinates {
    (9.1,32.68930775)
};

% BOX-42-05: mínimo en p = 13 y p = 14
\addplot[
    only marks,
    forget plot,
    mark=pentagon*,
    mark size=2.1pt,
    mark options={
        solid,
        fill=boxpurple!65!black,
        draw=boxpurple!65!black
    }
]
coordinates {
    (13.2,27.17729037)
    (14.2,27.17729037)
};

\end{axis}

\end{tikzpicture}

\vspace{0.35cm}

% =========================================================
% Legend below the plot
% =========================================================
\begin{tikzpicture}

% BOX-42-01
\draw[
    boxblue,
    line width=0.65pt
] (0,0) -- (0,0.34);
\fill[
    boxblue
] (0,0.34) circle[radius=1.5pt];
\node[
    anchor=west,
    font=\scriptsize
] at (0.18,0.17) {BOX-42-01};

% BOX-42-02
\begin{scope}[xshift=2.35cm]
    \draw[
        boxred,
        line width=0.65pt
    ] (0,0) -- (0,0.34);
    \fill[
        boxred
    ] (-0.055,0.285) rectangle (0.055,0.395);
    \node[
        anchor=west,
        font=\scriptsize
    ] at (0.18,0.17) {BOX-42-02};
\end{scope}

% BOX-42-03
\begin{scope}[xshift=4.70cm]
    \draw[
        boxgreen,
        line width=0.65pt
    ] (0,0) -- (0,0.34);
    \fill[
        boxgreen
    ] (0,0.42) -- (-0.075,0.28) -- (0.075,0.28) -- cycle;
    \node[
        anchor=west,
        font=\scriptsize
    ] at (0.18,0.17) {BOX-42-03};
\end{scope}

% BOX-42-04
\begin{scope}[xshift=7.05cm]
    \draw[
        boxorange,
        line width=0.65pt
    ] (0,0) -- (0,0.34);
    \fill[
        boxorange
    ] (0,0.43) -- (0.085,0.34) -- (0,0.25) -- (-0.085,0.34) -- cycle;
    \node[
        anchor=west,
        font=\scriptsize
    ] at (0.18,0.17) {BOX-42-04};
\end{scope}

% BOX-42-05
\begin{scope}[xshift=9.40cm]
    \draw[
        boxpurple,
        line width=0.65pt
    ] (0,0) -- (0,0.34);
    \fill[
        boxpurple
    ] 
        (0,0.43) --
        (0.082,0.37) --
        (0.052,0.27) --
        (-0.052,0.27) --
        (-0.082,0.37) --
        cycle;
    \node[
        anchor=west,
        font=\scriptsize
    ] at (0.18,0.17) {BOX-42-05};
\end{scope}

\end{tikzpicture}

\end{figure}

The minima in Figure~\ref{fig:p-obop} correspond to the numbers of buckets attained by optimal solutions of the standard {\rm OBOP}. Although the curves appear nearly unimodal in these instances, this behaviour is not guaranteed in general. In particular, the standard {\rm OBOP} may admit optimal solutions with two different numbers of buckets without admitting an optimal solution for every intermediate number of buckets, as shown in the following example.

\begin{example}\label{ex:ejemplo_B}
Consider the following {\rm OBOP} instance with \(4\) items:
\[
C=\frac{1}{100}
\begin{pmatrix}
50 & 55 & 90 & 100 \\
45 & 50 & 20 & 80 \\
10 & 80 & 50 & 65 \\
0 & 20 & 35 & 50
\end{pmatrix}.
\]
Solving the \(p\)-{\rm OBOP} for each possible number of buckets gives
\[
\begin{array}{c|cccc}
p & 1 & 2 & 3 & 4 \\
\hline
\text{Opt. value} & 3.40 & 2.60 & 2.80 & 2.60
\end{array}
\]
Therefore, the standard {\rm OBOP} has optimal solutions with \(2\) and \(4\) buckets, both attaining objective value \(2.60\), whereas every solution with the intermediate number \(p=3\) has objective value at least \(2.80\). Consequently, the numbers of buckets associated with optimal {\rm OBOP} solutions need not form a consecutive set. Equivalently, the optimal value of the \(p\)-{\rm OBOP} is not necessarily a unimodal function of \(p\).
\end{example}

\subsubsection{{\rm TCU}-\texorpdfstring{$k$}{k} {\rm OBOP}}

We now examine how the optimal value changes under the TCU-\(k\) formulation. To this end, we solve the same family of instances for several admissible values of \(k\), selected among \(3\), \(5\), \(\lceil n/10\rceil\), \(\lceil n/5\rceil\), and \(\lceil n/2\rceil\), using model~\eqref{mod:top-k} interpreted as the TCU-\(k\) {\rm OBOP}. In addition, we report the value \(n-|B_{\rm last}|\), where \(B_{\rm last}\) denotes the last bucket of an optimal solution of the standard {\rm OBOP}. When several optimal solutions have different last-bucket sizes, all corresponding values are reported.

\begingroup
\scriptsize
\setlength{\tabcolsep}{3pt}
\begin{longtable}{lrr rr rrrrr}
\caption{Results obtained for the TCU-\(k\) {\rm OBOP} model. The columns report the instance name, the number of items \(n\), the number of voters \(m\), the optimal value of the standard {\rm OBOP}, the value \(n-|B_{\rm last}|\), and the optimal values of the TCU-\(k\) {\rm OBOP} for several values of \(k\), whenever admissible. The symbol ``--'' indicates that the corresponding value of \(k\) is not admissible.}
\label{tab:TCU-k-OBOP}\\
\toprule
Instance & \(n\) & \(m\) & \multicolumn{2}{c}{{\rm OBOP}} & \multicolumn{5}{c}{TCU-\(k\) {\rm OBOP}} \\
\cmidrule(lr){4-5} \cmidrule(lr){6-10}
 & & & Opt. value & \(n-|B_{\rm last}|\) & \(k=3\) & \(k=5\) & \(k=\lceil n/10\rceil\) & \(k=\lceil n/5\rceil\) & \(k=\lceil n/2\rceil\) \\
\midrule
\endfirsthead

\caption[]{Results obtained for the TCU-\(k\) {\rm OBOP} model \textit{(continued)}.}\\
\toprule
Instance & \(n\) & \(m\) & \multicolumn{2}{c}{{\rm OBOP}} & \multicolumn{5}{c}{TCU-\(k\) {\rm OBOP}} \\
\cmidrule(lr){4-5} \cmidrule(lr){6-10}
 & & & Opt. value & \(n-|B_{\rm last}|\) & \(k=3\) & \(k=5\) & \(k=\lceil n/10\rceil\) & \(k=\lceil n/5\rceil\) & \(k=\lceil n/2\rceil\) \\
\midrule
\endhead

\midrule
\multicolumn{10}{r}{\textit{Continued on next page}}\\
\endfoot

\bottomrule
\endlastfoot

LPE-68-01 & 10 & 23 & 15.63 & 9 & 16.63 & 15.92 & 18.70 & 17.07 & 15.92 \\
LPE-68-02 & 14 & 15 & 34.14 & 13 & 36.14 & 41.14 & 35.14 & 36.14 & 42.14 \\
LPE-68-03 & 9 & 15 & 12.29 & 6 & 13.29 & 13.29 & 15.71 & 13.29 & 13.29 \\
LPE-68-04 & 17 & 15 & 29.78 & 16 & 43.67 & 51.11 & 40.44 & 47.11 & 56.11 \\
LPE-68-05 & 5 & 15 & 3.94 & 0 & 4.57 & -- & 4.60 & 4.60 & 4.57 \\
LPE-68-06 & 9 & 15 & 11.25 & 8 & 15.25 & 16.25 & 11.75 & 13.25 & 16.25 \\
LPE-68-07 & 13 & 15 & 25.73 & 5 & 29.26 & 25.73 & 31.26 & 29.26 & 30.26 \\
LPE-68-08 & 7 & 15 & 6.38 & 5 & 8.33 & 6.38 & 11.11 & 10.21 & 6.73 \\
LPE-68-09 & 10 & 21 & 6.99 & 8 & 17.49 & 9.55 & 27.46 & 22.49 & 9.55 \\
LPE-68-10 & 7 & 15 & 4.97 & 3 & 4.97 & 7.27 & 6.97 & 6.77 & 6.77 \\[0.2em]

BOX-42-01 & 17 & 31 & 14.27 & 16 & 80.58 & 56.74 & 94.32 & 67.58 & 30.27 \\
BOX-42-02 & 23 & 46 & 33.70 & 16 & 142.39 & 107.74 & 142.39 & 107.74 & 52.64 \\
BOX-42-03 & 20 & 36 & 22.94 & 19 & 108.59 & 78.59 & 125.59 & 94.37 & 42.71 \\
BOX-42-04 & 21 & 36 & 32.69 & 17 & 115.80 & 90.55 & 115.80 & 90.55 & 49.26 \\
BOX-42-05 & 22 & 37 & 27.18 & 15,21 & 139.68 & 108.11 & 139.68 & 108.11 & 41.02 \\
BOX-42-06 & 22 & 37 & 17.38 & 17 & 144.63 & 111.36 & 144.63 & 111.36 & 39.26 \\
BOX-42-07 & 25 & 41 & 51.58 & 20,21 & 188.07 & 161.83 & 188.07 & 161.83 & 79.10 \\
BOX-42-08 & 22 & 38 & 28.01 & 19 & 153.80 & 121.92 & 153.80 & 121.92 & 59.52 \\
BOX-42-09 & 18 & 9 & 3.89 & 16 & 98.33 & 73.11 & 113.33 & 86.11 & 32.89 \\
BOX-42-10 & 16 & 31 & 1.45 & 15 & 74.64 & 53.97 & 87.64 & 63.03 & 27.29 \\[0.2em]

SPR-56-01 & 31 & 5 & 76.07 & 24 & 306.87 & 263.87 & 283.87 & 226.87 & 117.67 \\
SPR-56-03 & 41 & 5 & 155.97 & 29 & 570.90 & 509.50 & 509.50 & 396.10 & 185.43 \\
SPR-56-08 & 35 & 5 & 83.13 & 10 & 411.73 & 360.53 & 385.93 & 314.73 & 128.53 \\
SPR-56-12 & 31 & 5 & 73.30 & 24,26 & 317.70 & 273.90 & 293.90 & 235.10 & 118.03 \\
SPR-56-15 & 29 & 5 & 62.17 & 17 & 274.63 & 236.03 & 274.63 & 219.03 & 93.43 \\
SPR-56-16 & 42 & 15 & 252.98 & 40 & 476.83 & 410.56 & 410.56 & 319.73 & 254.68 \\
SPR-56-19 & 55 & 15 & 402.95 & 28 & 853.27 & 764.83 & 730.55 & 570.81 & 407.43 \\
SPR-56-29 & 40 & 15 & 180.51 & 23 & 449.42 & 387.32 & 413.65 & 316.37 & 187.37 \\
SPR-56-30 & 43 & 14 & 266.56 & 23 & 496.68 & 460.83 & 460.83 & 385.58 & 269.20 \\
SPR-56-94 & 47 & 17 & 305.39 & 46 & 587.28 & 489.54 & 519.48 & 396.89 & 319.13 \\

\end{longtable}
\endgroup

Table~\ref{tab:TCU-k-OBOP} reports the optimal values obtained for the TCU-\(k\) {\rm OBOP}. When \(k\) is small, most items are forced into the collapsed tail bucket, and the resulting objective values are often much larger than those of the standard {\rm OBOP}, especially in the BOX and SPR instances. The column \(n-|B_{\rm last}|\) helps relate both models: whenever this value is admissible, an optimal {\rm OBOP} solution with last bucket \(B_{\rm last}\) is feasible for the TCU-\(k\) model with \(k=n-|B_{\rm last}|\). Hence, the standard {\rm OBOP} value is attained for those admissible values of \(k\).

From a computational perspective, the TCU-\(k\) model was generally solved within a few seconds, with the only notable outlier being instance SPR-56-19 for \(k=\lceil n/2\rceil\), which required \(307.77\) seconds. Figure~\ref{fig:tcu-k} provides a more detailed view for instances BOX-42-01 to BOX-42-05, showing the optimal value for every admissible value of \(k\). The global minima, highlighted in a darker shade, occur at values of \(k\) corresponding to the number of items outside the last bucket of an optimal {\rm OBOP} solution.

\begin{figure}
\centering

\caption{Objective values for the TCU-\(k\) {\rm OBOP} on instances BOX-42-01 to BOX-42-05, for every admissible value of \(k\).}
\label{fig:tcu-k}

\begin{tikzpicture}

\begin{axis}[
    width=0.98\linewidth,
    height=0.62\linewidth,
    title={TCU-$k$},
    xlabel={$k$},
    ylabel={Objective value},
    xmin=0.5, xmax=22.5,
    ymin=10, ymax=185,
    xtick={1,2,...,22},
    minor xtick={2,4,6,8,10,12,14,16,18,20,22},
    ytick={20,40,60,80,100,120,140,160,180},
    minor y tick num=1,
    grid=major,
    minor grid style={gray!15},
    grid style={gray!25},
    tick label style={font=\scriptsize},
    label style={font=\small},
    title style={font=\small},
    axis x line*=bottom,
    axis y line*=left,
    clip=true
]

\pgfplotsset{
    piruleta/.style={
        ycomb,
        solid,
        line width=0.65pt,
        mark=*,
        mark size=1.5pt,
        mark options={solid}
    }
}

% BOX-42-01
\addplot[
    piruleta,
    color=boxblue,
    mark=*
]
table[
    x expr=\thisrow{k}-0.2,
    y=val
] {
k val
1 109.3234185
2 94.32341852
3 80.58148303
4 67.58148303
5 56.74277336
6 47.67825723
7 37.54277336
8 34.44753526
9 30.27004608
10 27.27004608
11 23.27004608
12 20.27004608
13 18.27004608
14 17.27004608
15 15.27004608
16 14.27004608
};

% BOX-42-02
\addplot[
    piruleta,
    color=boxred,
    mark=square*
]
table[
    x expr=\thisrow{k}-0.1,
    y=val
] {
k val
1 176.3997525
2 158.1997525
3 142.3910568
4 125.7997525
5 107.7388829
6 92.40844811
7 79.53277243
8 69.11295261
9 65.21250012
10 61.23342189
11 56.75120845
12 52.63805833
13 48.15444329
14 42.5546152
15 37.1453322
16 33.70358681
17 35.70358681
18 37.70358681
19 39.70358681
20 39.70358681
21 37.70358681
22 35.70358681
};

% BOX-42-03
\addplot[
    piruleta,
    color=boxgreen,
    mark=triangle*
]
table[
    x expr=\thisrow{k},
    y=val
] {
k val
1 142.5940637
2 125.5940637
3 108.5940637
4 94.37184146
5 78.59406368
6 67.22073034
7 60.83184146
8 54.99184146
9 49.49184146
10 42.71406368
11 34.69850812
12 33.19313178
13 31.15012103
14 30.26250198
15 28.81088908
16 27.14422241
17 24.94422241
18 24.39876786
19 22.94422241
};

% BOX-42-04
\addplot[
    piruleta,
    color=boxorange,
    mark=diamond*
]
table[
    x expr=\thisrow{k}+0.1,
    y=val
] {
k val
1 147.8738108
2 128.8738108
3 115.7968878
4 102.1302211
5 90.5489357
6 82.79294442
7 74.20470913
8 66.61647383
9 58.86969991
10 53.28146461
11 49.25967812
12 45.22264108
13 42.99649729
14 39.94521524
15 36.83410413
16 33.30469237
17 32.68930775
18 33.30469237
19 34.68930775
20 33.68930775
};

% BOX-42-05
\addplot[
    piruleta,
    color=boxpurple,
    mark=pentagon*
]
table[
    x expr=\thisrow{k}+0.2,
    y=val
] {
k val
1 177.9220494
2 158.6788061
3 139.6788061
4 122.9761034
5 108.1112386
6 91.54014574
7 76.88797182
8 65.76928792
9 54.97492834
10 46.36623269
11 41.01647117
12 36.42573758
13 32.91835624
14 28.5270519
15 27.17729037
16 28.5270519
17 29.17729037
18 29.17729037
19 29.17729037
20 29.17729037
21 27.17729037
};

% =========================================================
% Mínimos ligeramente más visibles
% =========================================================

% BOX-42-01: mínimo en k = 16
\addplot[
    only marks,
    forget plot,
    mark=*,
    mark size=2.0pt,
    mark options={
        solid,
        fill=boxblue!65!black,
        draw=boxblue!65!black
    }
]
coordinates {
    (15.8,14.27004608)
};

% BOX-42-02: mínimo en k = 16
\addplot[
    only marks,
    forget plot,
    mark=square*,
    mark size=2.0pt,
    mark options={
        solid,
        fill=boxred!65!black,
        draw=boxred!65!black
    }
]
coordinates {
    (15.9,33.70358681)
};

% BOX-42-03: mínimo en k = 19
\addplot[
    only marks,
    forget plot,
    mark=triangle*,
    mark size=2.1pt,
    mark options={
        solid,
        fill=boxgreen!65!black,
        draw=boxgreen!65!black
    }
]
coordinates {
    (19,22.94422241)
};

% BOX-42-04: mínimo en k = 17
\addplot[
    only marks,
    forget plot,
    mark=diamond*,
    mark size=2.1pt,
    mark options={
        solid,
        fill=boxorange!65!black,
        draw=boxorange!65!black
    }
]
coordinates {
    (17.1,32.68930775)
};

% BOX-42-05: mínimo en k = 15 y k = 21
\addplot[
    only marks,
    forget plot,
    mark=pentagon*,
    mark size=2.1pt,
    mark options={
        solid,
        fill=boxpurple!65!black,
        draw=boxpurple!65!black
    }
]
coordinates {
    (15.2,27.17729037)
    (21.2,27.17729037)
};

\end{axis}

\end{tikzpicture}

\vspace{0.35cm}

% =========================================================
% Legend below the plot
% =========================================================
\begin{tikzpicture}

% BOX-42-01
\draw[
    boxblue,
    line width=0.65pt
] (0,0) -- (0,0.34);
\fill[
    boxblue
] (0,0.34) circle[radius=1.5pt];
\node[
    anchor=west,
    font=\scriptsize
] at (0.18,0.17) {BOX-42-01};

% BOX-42-02
\begin{scope}[xshift=2.35cm]
    \draw[
        boxred,
        line width=0.65pt
    ] (0,0) -- (0,0.34);
    \fill[
        boxred
    ] (-0.055,0.285) rectangle (0.055,0.395);
    \node[
        anchor=west,
        font=\scriptsize
    ] at (0.18,0.17) {BOX-42-02};
\end{scope}

% BOX-42-03
\begin{scope}[xshift=4.70cm]
    \draw[
        boxgreen,
        line width=0.65pt
    ] (0,0) -- (0,0.34);
    \fill[
        boxgreen
    ] (0,0.42) -- (-0.075,0.28) -- (0.075,0.28) -- cycle;
    \node[
        anchor=west,
        font=\scriptsize
    ] at (0.18,0.17) {BOX-42-03};
\end{scope}

% BOX-42-04
\begin{scope}[xshift=7.05cm]
    \draw[
        boxorange,
        line width=0.65pt
    ] (0,0) -- (0,0.34);
    \fill[
        boxorange
    ] (0,0.43) -- (0.085,0.34) -- (0,0.25) -- (-0.085,0.34) -- cycle;
    \node[
        anchor=west,
        font=\scriptsize
    ] at (0.18,0.17) {BOX-42-04};
\end{scope}

% BOX-42-05
\begin{scope}[xshift=9.40cm]
    \draw[
        boxpurple,
        line width=0.65pt
    ] (0,0) -- (0,0.34);
    \fill[
        boxpurple
    ]
        (0,0.43) --
        (0.082,0.37) --
        (0.052,0.27) --
        (-0.052,0.27) --
        (-0.082,0.37) --
        cycle;
    \node[
        anchor=west,
        font=\scriptsize
    ] at (0.18,0.17) {BOX-42-05};
\end{scope}

\end{tikzpicture}

\end{figure}

\subsubsection{{Fair \rm OBOP}}

We complete the computational tests with the Fair {\rm OBOP}. We consider two illustrative group partitions. The first one is an index-threshold partition, separating the first \(80\%\) of the items, according to their indices, from the remaining \(20\%\):
\[
G_1=\{\,r\in[n]: r\le 0.8n\,\},
\qquad
G_2=\{\,r\in[n]: r>0.8n\,\}.
\]
The second one is a residue-class partition, which groups the items according to their index modulo three:
\[
G_1=\{\,r\in[n]: r\bmod 3=0\,\},\quad
G_2=\{\,r\in[n]: r\bmod 3=1\,\},\quad
G_3=\{\,r\in[n]: r\bmod 3=2\,\}.
\]
In both settings, proportional bounds are imposed with
\[
\lambda_{i\ell}=\mu_{i\ell}=\frac{|G_i|}{n},
\]
and the number of buckets is restricted to lie between the minimum and maximum number of buckets attained by optimal solutions of the standard {\rm OBOP}.

\begingroup
\scriptsize
\setlength{\tabcolsep}{5pt}
\begin{longtable}{lrr rrr}
\caption{Results obtained for the Fair {\rm OBOP} model. The columns report the instance name, the number of items \(n\), the number of voters \(m\), the optimal value of the standard {\rm OBOP}, and the optimal values obtained under the index-threshold and residue-class partitions.}
\label{tab:Fairness-OBOP}\\
\toprule
Instance & \(n\) & \(m\) & Opt. OBOP & \multicolumn{2}{c}{Fair {\rm OBOP}} \\
\cmidrule(lr){5-6}
 & & & & Index-threshold partition & Residue-class partition \\
\midrule
\endfirsthead

\caption[]{Results obtained for the Fair {\rm OBOP} model \textit{(continued)}.}\\
\toprule
Instance & \(n\) & \(m\) & Opt. OBOP & \multicolumn{2}{c}{Fair {\rm OBOP}} \\
\cmidrule(lr){5-6}
 & & & & Index-threshold partition & Residue-class partition \\
\midrule
\endhead

\midrule
\multicolumn{6}{r}{\textit{Continued on next page}}\\
\endfoot

\bottomrule
\endlastfoot

LPE-68-01 & 10 & 23 & 15.63 & 15.63 & 15.70 \\
LPE-68-02 & 14 & 15 & 34.14 & 34.14 & 34.14 \\
LPE-68-03 & 9 & 15 & 12.29 & 13.29 & 12.29 \\
LPE-68-04 & 17 & 15 & 29.78 & 29.78 & 29.78 \\
LPE-68-05 & 5 & 15 & 3.94 & 3.94 & 3.94 \\
LPE-68-06 & 9 & 15 & 11.25 & 11.25 & 11.25 \\
LPE-68-07 & 13 & 15 & 25.73 & 25.73 & 30.73 \\
LPE-68-08 & 7 & 15 & 6.38 & 6.53 & 7.56 \\
LPE-68-09 & 10 & 21 & 6.99 & 6.99 & 9.14 \\
LPE-68-10 & 7 & 15 & 4.97 & 4.97 & 4.97 \\[0.2em]

BOX-42-01 & 17 & 31 & 14.27 & 14.27 & 24.68 \\
BOX-42-02 & 23 & 46 & 33.70 & 34.72 & 34.83 \\
BOX-42-03 & 20 & 36 & 22.94 & 24.72 & 25.72 \\
BOX-42-04 & 21 & 36 & 32.69 & 38.28 & 34.23 \\
BOX-42-05 & 22 & 37 & 27.18 & 39.73 & 36.73 \\
BOX-42-06 & 22 & 37 & 17.38 & 42.78 & 29.57 \\
BOX-42-07 & 25 & 41 & 51.58 & 51.58 & 60.79 \\
BOX-42-08 & 22 & 38 & 28.01 & 29.64 & 42.13 \\
BOX-42-09 & 18 & 9 & 3.89 & 10.22 & 12.89 \\
BOX-42-10 & 16 & 31 & 1.45 & 6.68 & 42.32 \\[0.2em]

SPR-56-01 & 31 & 5 & 76.07 & 90.87 & 93.27 \\
SPR-56-03 & 41 & 5 & 155.97 & 162.57 & 176.83 \\
SPR-56-08 & 35 & 5 & 83.13 & 86.60 & 130.27 \\
SPR-56-12 & 31 & 5 & 73.30 & 84.17 & 92.03 \\
SPR-56-15 & 29 & 5 & 62.17 & 63.10 & 68.83 \\
SPR-56-16 & 42 & 15 & 252.98 & 258.41 & 268.87 \\
SPR-56-19 & 55 & 15 & 402.95 & 406.90 & 431.70 \\
SPR-56-29 & 40 & 15 & 180.51 & 182.59 & 186.59 \\
SPR-56-30 & 43 & 14 & 266.56 & 266.56 & 278.56 \\
SPR-56-94 & 47 & 17 & 305.39 & 308.46 & 310.37 \\

\end{longtable}
\endgroup

Table~\ref{tab:Fairness-OBOP} shows that, in several instances, the Fair {\rm OBOP} attains the same value as the standard {\rm OBOP}, meaning that an optimal unconstrained bucket order already satisfies the imposed proportional bounds. In other cases, the objective value increases, reflecting the cost of enforcing group representation. This effect is particularly visible in some BOX instances, especially under the residue-class partition.

From a computational perspective, the Fair {\rm OBOP} is clearly more demanding than the previous variants. The index-threshold partition was solved with a maximum time of \(17316.02\) seconds, whereas the residue-class partition reached \(76209.18\) seconds in the hardest instance. The largest times occur mainly in the SPR instances, while the LPE instances remain easy to solve. Thus, fairness constraints can be incorporated into the framework, but they may significantly increase the computational effort.

Figure~\ref{fig:fairness} illustrates the effect of the fairness constraints for instance BOX-42-10 under the residue-class partition. The horizontal axis represents the prefix index \(\ell\), while the vertical axis shows, for each group, its proportion among the items contained in the first \(\ell\) buckets. In the standard {\rm OBOP} solution, Group~1 is absent from the early buckets and Group~2 is overrepresented. The fairness constraints enforce earlier representation of Group~1 and reduce the dominance of Group~2, while Group~3 is already close to proportional representation.

\begin{figure}
\centering
\caption{Comparison of group proportions with and without fairness for instance BOX-42-10}
\label{fig:fairness}
\def\lollybase{-0.03}

\pgfplotsset{
    fairnessAxis/.style={
        width=0.93\linewidth,
        height=0.3\linewidth,
        xmin=0.5,
        xmax=13.5,
        ymin=-0.05,
        ymax=1.05,
        xtick={1,2,...,13},
        ytick={0,0.2,0.4,0.6,0.8,1},
        grid=major,
        grid style={gray!10, line width=0.25pt},
        axis x line*=bottom,
        axis y line*=left,
        axis line style={gray!55, line width=0.45pt},
        tick style={gray!55, line width=0.45pt},
        xtick pos=bottom,
        ytick pos=left,
        tick label style={font=\scriptsize},
        label style={font=\scriptsize},
        ylabel style={font=\scriptsize},
        title style={font=\small},
        clip=true
    },
    piruletaWarm/.style={
        ycomb,
        solid,
        line width=0.85pt,
        color=orange!85!red,
        mark=*,
        mark size=2.05pt,
        mark options={
            solid,
            fill=orange!85!red,
            draw=orange!85!red
        }
    },
    piruletaCold/.style={
        ycomb,
        solid,
        line width=0.85pt,
        color=cyan!70!blue,
        mark=*,
        mark size=2.05pt,
        mark options={
            solid,
            fill=cyan!70!blue,
            draw=cyan!70!blue
        }
    },
    targetLine/.style={
        gray!55,
        densely dotted,
        line width=0.9pt
    },
    lowerBound/.style={
        black,
        dashed,
        line width=0.8pt,
        mark=none
    },
    upperBound/.style={
        black,
        densely dashdotted,
        line width=0.8pt,
        mark=none
    }
}

\begin{tikzpicture}

\begin{groupplot}[
    group style={
        group size=1 by 3,
        vertical sep=1.5cm
    },
    fairnessAxis
]

% =========================================================
% Group G_1
% =========================================================
\nextgroupplot[
    title={Group $G_1$},
    ylabel={Proportion of group $G_1$}
]

% Ideal proportion |G_1|/n = 5/16
\addplot[
    targetLine,
    forget plot
]
coordinates {
    (0.5,{5/16})
    (13.5,{5/16})
};

% Lower bound: floor(|B_l||G_1|/n)/|B_l|
\addplot[
    lowerBound,
    const plot,
    forget plot
]
table[
    x=xleft,
    y expr=floor(\thisrow{den}*5/16)/\thisrow{den}
] {
xleft den
0.5 1
1.5 2
2.5 3
3.5 6
4.5 7
5.5 8
6.5 9
7.5 11
8.5 12
9.5 13
10.5 14
11.5 15
12.5 16
13.5 16
};

% Upper bound: ceil(|B_l||G_1|/n)/|B_l|
\addplot[
    upperBound,
    const plot,
    forget plot
]
table[
    x=xleft,
    y expr=ceil(\thisrow{den}*5/16)/\thisrow{den}
] {
xleft den
0.5 1
1.5 2
2.5 3
3.5 6
4.5 7
5.5 8
6.5 9
7.5 11
8.5 12
9.5 13
10.5 14
11.5 15
12.5 16
13.5 16
};

% Without fairness, G_1
\addplot[
    piruletaWarm
]
table[
    x expr=\thisrow{x}-0.07,
    y expr=\thisrow{num}/\thisrow{den}
] {
x den num
1 1 0
2 2 0
3 3 0
4 4 0
5 6 0
6 8 0
7 10 0
8 11 1
9 12 1
10 13 2
11 14 3
12 15 4
13 16 5
};

% With fairness, G_1
\addplot[
    piruletaCold
]
table[
    x expr=\thisrow{x}+0.07,
    y expr=\thisrow{num}/\thisrow{den}
] {
x den num
1 1 0
2 2 0
3 3 0
4 6 1
5 7 2
6 8 2
7 9 2
8 11 3
9 12 3
10 13 4
11 14 4
12 15 4
13 16 5
};

% =========================================================
% Group G_2
% =========================================================
\nextgroupplot[
    title={Group $G_2$},
    ylabel={Proportion of group $G_2$}
]

% Ideal proportion |G_2|/n = 6/16
\addplot[
    targetLine,
    forget plot
]
coordinates {
    (0.5,{6/16})
    (13.5,{6/16})
};

% Lower bound: floor(|B_l||G_2|/n)/|B_l|
\addplot[
    lowerBound,
    const plot,
    forget plot
]
table[
    x=xleft,
    y expr=floor(\thisrow{den}*6/16)/\thisrow{den}
] {
xleft den
0.5 1
1.5 2
2.5 3
3.5 6
4.5 7
5.5 8
6.5 9
7.5 11
8.5 12
9.5 13
10.5 14
11.5 15
12.5 16
13.5 16
};

% Upper bound: ceil(|B_l||G_2|/n)/|B_l|
\addplot[
    upperBound,
    const plot,
    forget plot
]
table[
    x=xleft,
    y expr=ceil(\thisrow{den}*6/16)/\thisrow{den}
] {
xleft den
0.5 1
1.5 2
2.5 3
3.5 6
4.5 7
5.5 8
6.5 9
7.5 11
8.5 12
9.5 13
10.5 14
11.5 15
12.5 16
13.5 16
};

% Without fairness, G_2
\addplot[
    piruletaWarm
]
table[
    x expr=\thisrow{x}-0.07,
    y expr=\thisrow{num}/\thisrow{den}
] {
x den num
1 1 1
2 2 2
3 3 2
4 4 3
5 6 4
6 8 5
7 10 6
8 11 6
9 12 6
10 13 6
11 14 6
12 15 6
13 16 6
};

% With fairness, G_2
\addplot[
    piruletaCold
]
table[
    x expr=\thisrow{x}+0.07,
    y expr=\thisrow{num}/\thisrow{den}
] {
x den num
1 1 1
2 2 1
3 3 2
4 6 3
5 7 3
6 8 3
7 9 4
8 11 5
9 12 5
10 13 5
11 14 6
12 15 6
13 16 6
};

% =========================================================
% Group G_3
% =========================================================
\nextgroupplot[
    title={Group $G_3$},
    xlabel={Top-$\ell$},
    ylabel={Proportion of group $G_3$}
]

% Ideal proportion |G_3|/n = 5/16
\addplot[
    targetLine,
    forget plot
]
coordinates {
    (0.5,{5/16})
    (13.5,{5/16})
};

% Lower bound: floor(|B_l||G_3|/n)/|B_l|
\addplot[
    lowerBound,
    const plot,
    forget plot
]
table[
    x=xleft,
    y expr=floor(\thisrow{den}*5/16)/\thisrow{den}
] {
xleft den
0.5 1
1.5 2
2.5 3
3.5 6
4.5 7
5.5 8
6.5 9
7.5 11
8.5 12
9.5 13
10.5 14
11.5 15
12.5 16
13.5 16
};

% Upper bound: ceil(|B_l||G_3|/n)/|B_l|
\addplot[
    upperBound,
    const plot,
    forget plot
]
table[
    x=xleft,
    y expr=ceil(\thisrow{den}*5/16)/\thisrow{den}
] {
xleft den
0.5 1
1.5 2
2.5 3
3.5 6
4.5 7
5.5 8
6.5 9
7.5 11
8.5 12
9.5 13
10.5 14
11.5 15
12.5 16
13.5 16
};

% Without fairness, G_3
\addplot[
    piruletaWarm
]
table[
    x expr=\thisrow{x}-0.07,
    y expr=\thisrow{num}/\thisrow{den}
] {
x den num
1 1 0
2 2 0
3 3 1
4 4 1
5 6 2
6 8 3
7 10 4
8 11 4
9 12 5
10 13 5
11 14 5
12 15 5
13 16 5
};

% With fairness, G_3
\addplot[
    piruletaCold
]
table[
    x expr=\thisrow{x}+0.07,
    y expr=\thisrow{num}/\thisrow{den}
] {
x den num
1 1 0
2 2 1
3 3 1
4 6 2
5 7 2
6 8 3
7 9 3
8 11 3
9 12 4
10 13 4
11 14 4
12 15 5
13 16 5
};

\end{groupplot}

\end{tikzpicture}

\vspace{0.35cm}

% =========================================================
% Legend below the plots
% =========================================================
\begin{tikzpicture}

% Parámetros
\def\L{0.65}      % longitud de la línea de muestra
\def\T{0.82}      % distancia desde la línea al texto
\def\sepA{2.80}   % semiseparación fila 1
\def\sepB{4.20}   % semiseparación fila 2

% =====================
% First row (centered)
% =====================

% Without fairness
\begin{scope}[xshift=-\sepA cm, yshift=0.45cm]
    \draw[
        orange!85!red,
        line width=0.85pt
    ] (0,0.17) -- (\L,0.17);
    \fill[
        orange!85!red
    ] (0.325,0.17) circle[radius=2.05pt];
    \node[
        anchor=mid west,
        font=\scriptsize
    ] at (\T,0.17) {Without fairness};
\end{scope}

% With fairness
\begin{scope}[xshift=\sepA cm, yshift=0.45cm]
    \draw[
        cyan!70!blue,
        line width=0.85pt
    ] (0,0.17) -- (\L,0.17);
    \fill[
        cyan!70!blue
    ] (0.325,0.17) circle[radius=2.05pt];
    \node[
        anchor=mid west,
        font=\scriptsize
    ] at (\T,0.17) {With fairness};
\end{scope}

% =====================
% Second row (centered)
% =====================

% Ideal proportion
\begin{scope}[xshift=-\sepB cm, yshift=0cm]
    \draw[
        gray!55,
        densely dotted,
        line width=0.9pt
    ] (0,0.17) -- (\L,0.17);
    \node[
        anchor=mid west,
        font=\scriptsize
    ] at (\T,0.17) {Ideal proportion};
\end{scope}

% Lower bound
\begin{scope}[xshift=0cm, yshift=0cm]
    \draw[
        black,
        dashed,
        line width=0.8pt
    ] (0,0.17) -- (\L,0.17);
    \node[
        anchor=mid west,
        font=\scriptsize
    ] at (\T,0.17) {Lower bound};
\end{scope}

% Upper bound
\begin{scope}[xshift=\sepB cm, yshift=0cm]
    \draw[
        black,
        densely dashdotted,
        line width=0.8pt
    ] (0,0.17) -- (\L,0.17);
    \node[
        anchor=mid west,
        font=\scriptsize
    ] at (\T,0.17) {Upper bound};
\end{scope}

\end{tikzpicture}

\end{figure}

\section{Conclusions and future research} \label{sec:CFR}

This paper has introduced an optimization-based framework for rank aggregation problems whose consensus output is a weak order, or bucket order. The framework is based on pairwise precedence-or-tie variables and provides a flexible modelling structure in which different aggregation objectives and additional requirements can be incorporated through linear constraints. Within this framework, we have presented the first exact approach for the Optimal Bucket Order Problem ({\rm OBOP}).

The proposed framework has also been extended to several constrained weak-order aggregation settings. In particular, we have developed formulations with a prescribed number or size of buckets, introduced a model that preserves the weak-order structure of the upper part of the ranking while collapsing the remaining items into a final bucket, and incorporated group-based fairness requirements through cumulative constraints over prefixes of buckets. These extensions show that the weak-order structure can be combined with practical modelling requirements without abandoning the exact optimization setting.

The computational results clarify the practical value of the  formulation. Previous heuristic approaches for the {\rm OBOP} had produced high-quality solutions, but their optimality could not be certified. Our model shows that, for most benchmark instances, those solutions are in fact globally optimal, while also identifying instances in which the best known heuristic value can be improved. Therefore, the exact formulation provides more than an alternative way of solving the problem: it supplies optimality certificates, validates existing heuristic evidence, and creates exact reference values for future algorithmic developments. The experiments with constrained variants further illustrate how structural and fairness requirements may alter the optimal bucket order and quantify the objective-value cost associated with imposing such requirements.

Several research directions remain open. From a computational perspective, larger instances could be addressed through stronger formulations, cutting-plane strategies, decomposition methods, or hybrid exact-heuristic approaches. In particular, high-quality heuristic or metaheuristic solutions could be used as warm starts for the proposed formulations, especially in the most challenging instances. From a modelling perspective, future work could study alternative distance measures, richer forms of incomplete or uncertain input rankings, and additional families of fairness or diversity constraints. It would also be interesting to analyze the sensitivity of the resulting weak orders with respect to the imposed structural and fairness parameters.

\section*{Acknowledgments}
{\sloppy
C. Domínguez and M. Landete were supported by project PID2021-122344NB-I00, funded by MCIN/AEI/10.13039/501100011033 and by “ERDF A way of making Europe”. C. Domínguez was also supported by project PID2024-156594NB-C21, funded by MCIU/AEI/10.13039/501100011033 and by ERDF/EU, and by Generalitat Valenciana under project Emergent CIGE/2024/132. J.D. Jaime-Alcántara and M. Landete were also supported by grant CIPROM/2024/34, funded by the Conselleria de Educación, Cultura, Universidades y Empleo, Generalitat Valenciana. J.A. Aledo was supported by SBPLY/21/180225/000062 (Junta de Comunidades de Castilla-La Mancha and ERDF A way of making Europe), PID2022-139293NB-C32 (MICIU/AEI/10.13039/501100011033 and ERDF, EU) and 2025-GRIN-38476 (Universidad de Castilla-La Mancha and ERDF A way of making Europe).\par}

% --- Declarations Section ---
\section*{Declarations}

\paragraph{Conflict of interest} 
The authors declare that they have no known competing financial interests or personal relationships that could have appeared to influence the work reported in this paper.

\paragraph{Data and code availability}
The benchmark instances used in this study are derived from data obtained from the PrefLib and MovieLens datasets. Both the instances and the source code for the exact and matheuristic models are publicly available in Zenodo at DOI: \href{https://doi.org/10.5281/zenodo.18799994}{10.5281/zenodo.18799994}.





\bibliography{bibliography}% common bib file
%%\input article.bbl

\end{document}